\documentclass[11pt]{article}

\usepackage{amsmath,amssymb,amsthm}
\usepackage{color}
\usepackage{graphicx}
\newtheorem{theorem}{Theorem}

\newtheorem{definition}{Definition}
\newtheorem{lemma}{Lemma}
\newtheorem{cor}{Corollary}

\newtheorem{proposition}{Proposition}
\newtheorem{remark}{Remark}
\newtheorem{property}{Property}
\newtheorem{example}{Example}
\def \beq{ \begin{equation}}
\def \eeq{\end{equation}}

\newcommand{\rank}{\mathrm{rank}}

\newcommand{\olre}{$0$\,-$LRE$}
\newcommand{{\oere}}{$0$\,-$ERE$}

\title{Continuations and bifurcations of relative equilibria for the positive curved three body problem}
\date{} 
\begin{document}
	\maketitle
	\author{\begin{center}
	{ Toshiaki~Fujiwara$^1$, Ernesto P\'{e}rez-Chavela$^2$}\\	
		\bigskip
	   $^1$College of Liberal Arts and Sciences, Kitasato University,       Japan. fujiwara@kitasato-u.ac.jp\\
	    $^2$Department of Mathematics, ITAM, M\'exico.\\ ernesto.perez@itam.mx
	\end{center}
	

\bigskip

\begin{abstract}
The positive curved three body problem is a natural extension of the planar Newtonian three body problem to the sphere 
$\mathbb{S}^2$. In this paper we study the extensions of the Euler and Lagrange Relative equilibria ($RE$ in short) on the plane to the sphere.

The $RE$ on $\mathbb{S}^2$ are not isolated in general.
They usually have one-dimensional continuation in the three-dimensional shape space.
We show that there are two types of bifurcations. One is the bifurcations between 
Lagrange $RE$  and Euler $RE$. Another one is between the different types of the shapes of Lagrange $RE$. We prove that
bifurcations between equilateral and isosceles Lagrange $RE$  exist
for equal masses case, and that bifurcations between isosceles and scalene Lagrange $RE$  exist for partial equal masses case. 
\end{abstract}


{\bf Math. Subject Class 2020:} 70F07, 70F10, 70F15


\section{Introduction}
A relative equilibrium for the Newtonian $n$--body problem, is a particular solution where the masses are rotating uniformly around their center of mass with the same angular velocity. In these kind of motions the masses preserve their mutual distances for all time, that is, they behave as a rigid body motion. The fixed configuration at any time is called a central configuration. In the corresponding rotating frame they form an equilibrium point, from here the name \cite{Wintner, Moeckel}.

Since the central configurations are invariant under rotations and homotheties, we count classes of central configurations modulo these Euclidean transformations; then for $n=3$, there are only 5 classes of relative equilibria, three collinear or Euler relative equilibria \cite{Euler} and two equilateral triangle or Lagrange relative equilibrium \cite{Lagrange}.

\medskip

The curved $n$--body problem is a natural extension of the classical Newtonian  problem to spaces of constant curvature $\kappa$ which could be positive or negative. For $n=2$, the problem is divided into two classes, the Kepler problem (one particle is fixed, and the other one is moving according to it) and the $2$--body problem (both masses are moving according to their mutual attractions). The first one is an old problem, introduced independently in the 1830's by J. Bolyai and N. Lovachevsky the codiscovers  of the first non-Euclidean geometries. 
The second one was introduced by Borisov et al \cite{Borisov1, Borisov2}. Unlike the Newtonian problem, on the sphere these problems are not equivalent, the first one is integrable and the second one is not \cite{Shchepetilov2}. 

In 1994, V.V. Kozlov showed that the cotangent potential is the natural way to extend the Newtonian potential to spaces of constant positive curvature \cite{Kozlov}.
In 2012, F. Diacu, E. Perez-Chavela and M. Santoprete \cite{Diacu-EPC1}, obtained the generalization of this problem in an unified way for any value of $n$ and any value of the constant curvature $\kappa$. You can see \cite{Diacu1} and \cite{Borisov2} for a nice historical description of this problem.

In this paper we are interested in the analysis of relative equilibria for the two dimensional positive curved three body problem, which can be reduced to the analysis on the unit sphere $\mathbb{S}^2 \subset \mathbb{R}^3$. That is, we will study relative equilibria for three positive masses moving on $\mathbb{S}^2$ under the influence of the cotangent potential. From here on, just to simply the notation we call to the relative equilibria simply as $RE$.

By exploiting the symmetries of some configurations and using spherical trigonometric arguments, several authors have found different families of $RE$ on the sphere $\mathbb{S}^2$ see for instance \cite{Diacu-EPC1, Diacu1, Diacu3, zhu2, EPC1, EPC2, tibboel, zhu}.

In a recent paper \cite{Fujiwara}, the authors of this article developed a new systematic geometrical method to study $RE$ on $\mathbb{S}^2$, where the masses are moving on the sphere under the influence of a potential which only depends on the mutual distances among the masses, in particular for the cotangent potential. 
For $n=3$ the authors divide the analysis of $RE$ on $\mathbb{S}^2$ in two big classes, the Euler relative equilibria ($ERE$ by short) where the three masses are on the same geodesic, and the Lagrange relative equilibria ($LRE$ by short), for the $RE$ which are not in the previous class. In the same paper \cite{Fujiwara}, the authors find the necessary and sufficient conditions on the shapes, to obtain $ERE$ and $LRE$. In this paper we will restrict our analysis to the cotangent potential, that is, to the positive curved three body problem. 

In \cite{Bengochea}, the authors proved that any $RE$ of the planar $n$--body problem can be continued to spaces of constant curvature $\kappa$, positive or negative for small values of the parameter $\kappa$. 
This is a remarkable result, because for instance, it is well known that any three masses located at the vertices of an equilateral triangle generate a $RE$ on the plane; however in the case of curved spaces, $LRE$ with equilateral triangle shape only exist if the three masses are equal. In particular they show that any Lagrange relative equilibria can be continued to the sphere (the equilateral  triangle shape, is not preserved in the continuation if the masses are not equal). 
 
In this paper we will show that in general, we can continue a $RE$ on the sphere, represented by an one dimensional curve.
When two continuations of $RE$ intersect, we call it a bifurcation, in the next section we will give the precise definition of these concepts.

After the introduction, the paper is organized as follows: In Section 
\ref{shapes} we give the definitions of all concepts used along the paper, and we show how use the implicit function theorem to find the bifurcation points and the extensions of solutions. In Section \ref{equal masses} we do the analysis for the case of equal masses and in Section \ref{partial equal masses} we do the analysis for the case of partial equal masses. In section \ref{general masses}, we study $LRE$ with general masses, and finally in Section \ref{conclusions} we summarize our results and state some final comments.


\section{Conditions for a Shape}\label{shapes}
We consider three positive masses denoted by $m_k$ moving on a sphere of radius $R$ that we denote as $\mathbb{S}^2$.
The equations of motions are described by the Lagrangian, which in spherical coordinates $(R, \theta_k,\phi_k)$ take the form,
\begin{equation}
L=R^2\sum_k \frac{m_k}{2} (\dot{\theta}_k^2+\dot{\phi}_k^2\sin^2(\theta_k))
	+\sum_{i,j}\frac{m_i m_j}{R}
	\frac{\cos\sigma_{ij}}{\sqrt{1-\cos^2\sigma_{ij}}}.
\end{equation}
To facilitate the notations, we set $R=1$ unless otherwise specified. We denote the angle of the minor arc on the great circle connecting the masses $i$ and $j$ as $\sigma_{ij}$. In order to avoid singularities, we exclude the case $\cos^2\sigma_{ij} = 1$, which corresponds to 
$\sigma_{ij} \neq 0, \pi.$

The above angles are related to each other by the relationship
\begin{equation}\label{basicRelation}
\cos\sigma_{ij}=\cos\theta_i\cos\theta_j+\sin\theta_i\sin\theta_j\cos(\phi_i-\phi_j).
\end{equation}
For the three-body problem, it is convenient to use the notation
$\sigma_k=\sigma_{ij}$
for $(i,j,k)=(1,2,3),(2,3,1),$ and $(3,1,2)$.
We define
\begin{equation}\label{defOfU}
\begin{split}
U&=\{(\sigma_1,\sigma_2,\sigma_3)| 0<\sigma_k<\pi\},\\
U_\textrm{phys}&=\{(\sigma_1,\sigma_2,\sigma_3)\in U|
		\sigma_k\le \sigma_i+\sigma_j
		\mbox{ and }\sum_k \sigma_k\le 2\pi\}.
\end{split}
\end{equation}
The inequalities in $U_\textrm{phys}$
are the conditions of $\sigma_k$ to form a triangle.
Note that one point in $U_\textrm{phys}$
corresponds to two triangles with opposite orientation.

\begin{definition}\label{def:RE}
A relative equilibrium is a solution of the equations of motion, which in spherical coordinates satisfies 
$$\dot\theta_k=0 \quad \text{and} \quad \dot\phi_k=\omega \quad \text{ for all} \quad k=1,2,3.$$
That is, a solution that behaves as if the masses belonged to a rigid body, the shape is the same for all time.
\end{definition}

\begin{remark}
Usually, people working on {\it Geometric Mechanics} define the $RE$ as fixed points in a reduced system (see for instance \cite{Marsden} and the references therein). In other words, the $RE$ are solutions which are invariant under the action of a continuous symmetric group $G$. In our case the symmetric group is $G=SO(3)$. The $RE$ correspond to periodic orbits in the original phase space (the not reduced one), these periodic orbits are rotating uniformly around a principal axis. Then in order to determine a $RE$, we need to have the initial configuration and the angular velocity. If we express these conditions in the usual spherical coordinates, taking the rotation axis as the $z$--axis, we obtain Definition \ref{def:RE}. By the other hand, if Definition \ref{def:RE} holds, then since $G=SO(3)$, we obtain that a $RE$ is a fixed point in the reduced system. So, both definitions are equivalent.
\end{remark}

As we have seen in the previous Section,
there are two kinds a $RE$ on the sphere, $ERE$ and $LRE$.
In \cite{Fujiwara}, we proved that the $ERE$ must be on the equator or on a rotating meridian. 

The big difficulty to study $RE$ on the sphere is that the linear momentum and the center of mass are not more a first integral for the positive curved problem.

Fortunately in \cite{Fujiwara}, we found that the center of mass can be substitute by the vanishing of two components of the angular momentum on the sphere.

To verify this fact, we observe that the Lagrangian is invariant under rotations around the centre of $\mathbb{S}^2$, the angular momentum vector ${\bf c} = (c_x, c_y, c_z)$ is a first integral. The components of ${\bf c}$ are
\begin{align}
c_x&=R^2 \sum_k m_k \left(-\sin(\phi_k) \dot\theta_k
	-\sin(\theta_k)\cos(\theta_k)\cos(\phi_k) \dot\phi_k\right),
	\nonumber \\
c_y&=R^2 \sum_k m_k \left(\cos(\phi_k) \dot\theta_k
	-\sin(\theta_k)\cos(\theta_k)\sin(\phi_k) \dot\phi_k\right),
	\nonumber \\
c_z&=R^2 \sum_k m_k \sin^2(\theta_k) \dot \phi_k.
	\nonumber
\end{align}

Then, fixing the rotation axis as the $z$--axis, we obtain that the components $c_x=0$ and $c_y=0$ are integrals of motion.

Now by Definition \ref{def:RE},  after the substitution $\dot{\theta}_k=0$ and $\dot{\phi}_k(t) = \omega$, the angular momentum has the form 
${\bf c} = (c_x, c_y, c_z)=(0,0,c_z),$ where

\begin{align}\label{comp-am}
(c_x,c_y)&= -R^2 \omega \sum_k m_k \sin(\theta_k)\cos(\theta_k) \left(\cos(\phi_k),\sin(\phi_k)\right), \\
c_z&=R^2 \omega \sum_k m_k \sin^2(\theta_k) \nonumber. 
\end{align}

We observe that taking $r_k=R\theta_k$ finite, in 
the limit $R \to \infty$ we obtain $\theta_k \to 0$, since $\theta_k=r_k/R \to 0$ for $R\to \infty$.  Then, using equation \eqref{comp-am}, the expansion for $\cos \theta_k$ and $\sin \theta_k$ and dropping the higher order terms  we obtain
\begin{equation}\label{cmext}
c_x \to -R \omega \sum_k m_k r_k \cos(\phi_k)=0 \quad \text{and}
\quad  c_y  \to -R \omega \sum_k m_k r_k \sin(\phi_k)=0.
\end{equation}
Finally using the fact that $(c_x,c_y)=(0,0)$, the above equation \eqref{cmext}, is the condition to fix the center of mass at the origin 
on the Euclidean plane (see \cite{Fujiwara} for more details). 

\begin{remark}
In \cite{Borisov1} the authors use the reduction method to guarantee that the projections of the angular momentum of the system onto the original space are preserved. Then by using Noether's theorem, they obtain the expressions of the above projections in Euler's angles. Due to the more classical geometric viewpoint  that we discuss along this manuscript, we prefer to use our approach.
\end{remark}

Using the two integrals $c_x=0$ and $c_y=0$, we prove that the problem to find $RE$ on $\mathbb{S}^2$ is reduced to solve the eigenvalue problem $J\Psi_L = \lambda \Psi_L$ where $J$ is an useful representation of the inertia tensor given by (see \cite{Fujiwara} for more details) 

\begin{equation}\label{defJ}
J=\left(\begin{array}{ccc}
m_2+m_3 & -\sqrt{m_1m_2}\cos\sigma_{3} & -\sqrt{m_1m_3}\cos\sigma_{2} \\
-\sqrt{m_2m_1}\cos\sigma_{3} &m_3+m_1& -\sqrt{m_2m_3}\cos\sigma_{1} \\
-\sqrt{m_3m_1}\cos\sigma_{2} & -\sqrt{m_3m_2}\cos\sigma_{1} & m_1+m_2
\end{array}\right).
\end{equation}

In the same paper \cite{Fujiwara}, we show that the
necessary and sufficient conditions to have $LRE$ or $ERE$, 
can be described only in terms of $m_k$ and $\sigma_k$, $k=1,2,3$.
The conditions for a shape to form $LRE$ are
\begin{equation}\label{eqForLagrangeI}
\begin{split}
\lambda
&=\frac{(m_2+m_3)\sin^3(\sigma_1)
-m_2\cos(\sigma_3)\sin^3(\sigma_2)
-m_3\cos(\sigma_2)\sin^3(\sigma_3)}{\sin^3(\sigma_1)}
=\lambda_1\\
&=\frac{(m_3+m_1)\sin^3(\sigma_2)
-m_3\cos(\sigma_1)\sin^3(\sigma_3)
-m_1\cos(\sigma_3)\sin^3(\sigma_1)}{\sin^3(\sigma_2)}
=\lambda_2\\
&=\frac{(m_1+m_2)\sin^3(\sigma_3)
-m_1\cos(\sigma_2)\sin^3(\sigma_1)
-m_2\cos(\sigma_1)\sin^3(\sigma_2)}{\sin^3(\sigma_3)}
=\lambda_3.
\end{split}
\end{equation}

Let $\lambda_{ij}$ be the difference of $\lambda_i$ and $\lambda_j$,
namely
\begin{equation}
\lambda_{ij}=\lambda_i-\lambda_j.
\end{equation}
Then, the condition for have a $LRE$ is equivalent to 
$\lambda_{12}=\lambda_{23}=0$.

The correspondence between the solution $\sigma_k\in U_\textrm{phys}$
of this condition
and the configuration variables $\theta_k$ 
is given by
\begin{equation}\label{sigmaToTheta}
\cos\theta_k=s\sqrt{M-\lambda}\,\,\sin^3(\sigma_k)
	/\sqrt{\sum\nolimits_\ell m_\ell \sin^6(\sigma_\ell)}.
\end{equation}
Then using $\sigma_k$ and $\theta_k$,
equation (\ref{basicRelation}) determines
$\cos(\phi_i-\phi_j)$.
Where $M=\sum_\ell m_\ell$ is the total mass,
and $s=\pm 1$.
If we take $s=1$ the three masses are on the northern hemisphere,
when $s=-1$ they are on southern hemisphere.
Finally, the angular velocity is given by
\begin{equation}
\omega^2
=\frac{\sum_\ell m_\ell \sin^6(\sigma_\ell)}
{\sin^3(\sigma_1)\sin^3(\sigma_2)\sin^3(\sigma_3)}.
\end{equation}
Thus, the problem to find a configuration of $LRE$
is reduced to find the solution $\sigma_k$ of the condition
$\lambda_{12}=\lambda_{23}=0$ (See \cite{Fujiwara} for more details).

Similarly, the 
necessary and sufficient condition for a shape to form an $ERE$ 
on a rotating meridian with $\sigma_3=\sigma_1+\sigma_2$,
namely 
the mass $m_3$ is located between $m_1$ and $m_2$, is
\begin{equation}\label{eqForEREI}
\begin{split}
d=\frac{m_1 \sin(2\sigma_2)-m_2 \sin(2\sigma_1)}{\sin^2\sigma_3}
+&\frac{m_2 \sin(2\sigma_3)+m_3 \sin(2\sigma_2)}{\sin^2\sigma_1}\\
-&\frac{m_3 \sin(2\sigma_1)+m_1 \sin(2\sigma_3)}{\sin^2\sigma_2}
=0.
\end{split}
\end{equation}
You can consult
reference \cite{Fujiwara} for the correspondence
between $\sigma_k$ and 
the configuration variables $\theta_k$, $\omega^2$.

The conditions for a shape to form an $ERE$  on the equator 
$\sigma_1+\sigma_2+\sigma_3=2\pi$
are
\begin{equation}\label{ConditionEREonEquator}
\mu_k < \mu_i+\mu_j
\quad (i,j,k) \mbox{ is a permutation of }(1,2,3),
\end{equation}
where $\mu_k=\sqrt{m_im_j}$.
For this case, $\sigma_k$ is given by
\begin{equation}\label{SigmaForEREonEquator}
\sigma_k=\arccos\left(\frac{\mu_k^2-\mu_i^2-\mu_j^2}{2\mu_i\mu_j}\right).
\end{equation}
Note that there is just 
one set of $\sigma$'s 
(two shapes with different orientation) for given masses.
For the $LRE$ and the $ERE$ on a rotating meridian on the sphere, we will show that some of these $RE$ can be continued into the three dimensional space given by 
$(\sigma_1,\sigma_2,\sigma_3)$, 
and that these continuations can meet in this space.
We call this ``bifurcation'',
to be more precise we define:

\begin{definition}\label{defBifurcationPoint}
A bifurcation point of $LRE$ and $ERE$ on a rotating meridian is a point on a plane $\sigma_k=\sigma_i+\sigma_j$, where the continuation of a $LRE$ and an $ERE$ coincide. 
This bifurcation point is a coupling
where two continuations of $LRE$ with opposite orientation connected.
\end{definition}

We will show ahead in this paper, that some $LRE$ meets $ERE$ on the equator, and since an $ERE$ on the equator is just one point for given mass ratio and given orientation, we define:
\begin{definition}\label{defEulerJunction}
An Euler coupling on the Equator
is a point with $\sigma_1+\sigma_2+\sigma_3=2\pi$,
where two continuations of $LRE$ with the same orientation, one on the northern and the other one on the southern hemisphere, connected.
\end{definition}

Finally we have that the shapes of $LRE$ can be grouping into
equilateral, isosceles, and scalene triangles.
We will show that for equal masses case, $m_1=m_2=m_3$,
equilateral and isosceles $LRE$ have continuation.
When only two masses are equal, for instance $m_1=m_2\ne m_3$, we have continuation of isosceles and scalene $LRE$.
We define the bifurcation point
as follows.

\begin{definition}
Let the set $\{A,B\}$ be one of
\{equilateral, isosceles\} or
\{isosceles, scalene\} $LRE$.
A bifurcation point between $A$ and $B$ 
is the intersection point of the continuation of $A$  and $B$.
\end{definition}

The existence of bifurcation points between  
 $LRE$ one in the group $A$ and other in the group $B$
can be understood by two simple 
but important
properties of the surfaces defined by $\lambda_{ij}=0$
in $U_\textrm{phys}$.

\begin{property}
For $\sigma_i=\sigma_j$,
\begin{equation}
\lambda_{ij}=(m_i-m_j)(\cos\sigma_k-1).
\end{equation}
\end{property}

From this property we obtain the following proposition and corollaries whose proofs are obvious.

\begin{proposition}\label{equalSigmaNeedsEqualMass}
The necessary condition for $LRE$ with $\sigma_i=\sigma_j$ is $m_i=m_j.$
\end{proposition}

\begin{cor}\label{unequalMassesMakeScalene}
Unequal masses implies scalene triangle $LRE$.
\end{cor}
\begin{cor}
Equilateral triangle $LRE$ needs $m_1=m_2=m_3$.
\end{cor}

\begin{property}\label{factorisation}
For $m_i=m_j$, the function $\lambda_{ij}(\sigma_i,\sigma_j,\sigma_k)$ is an anti-symmetric function of $\sigma_i$ and $\sigma_j$,
and can be factorized as
\begin{equation}
\lambda_{ij}(\sigma_i,\sigma_j,\sigma_k)
=\frac{\sin(\sigma_i-\sigma_j)}{\sin^3(\sigma_i)\sin^3(\sigma_j)}\tilde\lambda_{ij}(\sigma_i,\sigma_j,\sigma_k),
\end{equation}
where
\begin{equation}\label{defTildeLambda}
\begin{split}
\tilde\lambda_{ij}(\sigma_i,\sigma_j,\sigma_k)
=&\nu_k \cos(\sigma_k)\sin(\sigma_i+\sigma_j)
	\Big(
		\sin^4(\sigma_i)+\sin^2(\sigma_i)\sin^2(\sigma_j)
		+\sin^4(\sigma_j)
	\Big)\\
&-\frac{\sin^3(\sigma_k)}{4}
	\Big(
		\cos(3\sigma_i+\sigma_j)+\cos(\sigma_i+3\sigma_j)
		-2\cos(\sigma_i+\sigma_j)
	\Big),\\
\text{and}
\quad \nu_k=\frac{m_i}{m_k}=\frac{m_j}{m_k}.
\end{split}
\end{equation}
\end{property}
This property explains the existence of  bifurcation points
between two groups of shapes of $LRE$.

We explain why for $m_1=m_2\ne m_3$ we obtain a bifurcation point of isosceles $LRE$ and scalene $LRE$.
For this case, the surface $\lambda_{12}=0$ is
split into the plane $\sigma_1=\sigma_2$ , and 
the surface
$\tilde\lambda_{12}=0$
by Property \ref{factorisation}. Then,
the condition for $LRE$, $\lambda_{12}=\lambda_{23}=0$,
is split into two cases,
intersection of $\sigma_1=\sigma_2$ and $\lambda_{23}=0$ namely
isosceles $LRE$, and
intersection of $\tilde\lambda_{12}=0$ and $\lambda_{23}=0$,
namely scalene $LRE$.
Then, the intersection of the three surfaces
$\sigma_1=\sigma_2$ and $\tilde\lambda_{12}=0$
and $\lambda_{23}=0$
gives the bifurcation points.

Similarly,
there are bifurcation points
of equilateral $LRE$ and isosceles $LRE$
for the equal masses case, $m_1=m_2=m_3$.

We will use the implicit function theorem to study bifurcations.
By this theorem,
if two continuous differentiable functions $f(x,y,z)$ and $g(x,y,z)$
have a solution $f(x_0,y_0,z_0)=g(x_0,y_0,z_0)=0$,
and $\nabla f \times \nabla g\ne 0$ at $(x_0,y_0,z_0)$,
then a continuation of the solution $f=g=0$  from this point
exists in some finite interval.
A geometrical interpretation of this theorem
is that
the conditions $f=g=0$ and $\nabla f \times \nabla g\ne 0$
means that the two surfaces intersect
at this point
and the vector $\nabla f \times \nabla g$
represents the tangent vector of the intersection curve at
this point.


\section{Equal masses case}\label{equal masses}
As shown in Property \ref{factorisation},
the conditions for have a $LRE$,
$\lambda_{12}=0$ and $\lambda_{23}=0$ 
are split into
($\sigma_1=\sigma_2$ or $\tilde\lambda_{12}=0$)
and ($\sigma_2=\sigma_3$ or $\tilde\lambda_{23}=0$).

First we will show that there are no scalene $LRE$
(intersection of $\tilde\lambda_{12}=0$
and $\tilde\lambda_{23}=0$),
then we describe the equilateral 
($\sigma_1=\sigma_2$ and $\sigma_2=\sigma_3$) and isosceles $LRE$ 
($\sigma_1=\sigma_2$ and  $\tilde\lambda_{23}=0$
or $\sigma_2=\sigma_3$ and $\tilde\lambda_{12}=0$)
and the bifurcation between them.

\subsection{No scalene $LRE$}
\begin{theorem}
There are no scalene $LRE$
for the equal masses case.
\end{theorem}
\begin{proof}
We will show a contradiction
if we assume that there is a scalene $LRE$ with
$(\sigma_1,\sigma_2,\sigma_3)=(x,y,z)$.
Let be $\tilde\lambda_{12}=g(\sigma_1,\sigma_2,\sigma_3)$,
then $\tilde\lambda_{23}=g(\sigma_2,\sigma_3,\sigma_1)$
and $\tilde\lambda_{31}=g(\sigma_3,\sigma_1,\sigma_2)$.
Since, this is a scalene triangle, $(x,y,z)$ must satisfies
$g(x,y,z)=g(y,z,x)=g(z,x,y)=0$.
By the definition of the function $g$, it is symmetric for the first two variables.
Therefore, $g$ must be zero for any permutation of $(x,y,z)$
if the assumption is true.

Now, define $g_0=g$ and $g_n$ for $n=1,2,3$ by
\begin{equation}
\begin{split}
g_0(x,y,z)-g_0(x,z,y)&=2\sin(y-z)g_1(x,y,z),\\
g_1(x,y,z)-g_1(y,x,z)
	&=-2(\cos x-\cos y)g_2(x,y,z),\\
g_2(x,y,z)-g_2(x,z,y)
	&=16(\cos y-\cos z)g_3(x,y,z).
\end{split}
\end{equation}
Since 
$g_0=0$ for any permutations of $(x,y,z)$ and $\sin(y-z)\ne 0$,
$g_1=0$ for any permutations of $(x,y,z)$.
Similarly, $g_2=g_3=0$.
Here, the function $g_3$ is a totally symmetric function of $(x,y,z)$,
\begin{equation}
\begin{split}
g_3(x,y,z)&=\Big(3-\cos(2x)-\cos(2y)-\cos(2z)\Big)\\
&\hspace{1cm}
	\cos\Big((x+y)/2\Big)\cos\Big((y+z)/2\Big)\cos\Big((z+x)/2\Big).
\end{split}
\end{equation}
The possible solutions of $g_3=0$ are
$x+y=\pi$ or $y+z=\pi$ or $z+x=\pi$.
However
\begin{equation}
g(x,y,z)|_{x+y=\pi}=\frac{1}{2}\Big(\cos(2x)-1\Big)\sin^3(z)\ne 0.
\end{equation}
Similarly
$g(y,z,x)|_{y+z=\pi}\ne 0$,
and
$g(z,x,y)|_{z+x=\pi}\ne 0$.
This is the contradiction we are looking for.
\end{proof}

\subsection{Isosceles and equilateral $LRE$}
%
%
%
%
%
%
In this subsection, we consider the shape
with $\sigma_1=\sigma_2=\sigma$.
So, 
\begin{equation}\label{UandU_phys}
U=(0,\pi)^2, \quad 
U_\textrm{phys}=\{(\sigma,\sigma_3)\in U \,\, | \,\,
\sigma_3\le 2\sigma \mbox{ and } \sigma_3\le 2(\pi-\sigma)\}.
\end{equation}

The isosceles solution is the solution of
\begin{equation}
\begin{split}
g(\sigma,\sigma_3)
=&g(\sigma_3,\sigma,\sigma)\\
=&\cos(\sigma)
	\Big(\sin^4(\sigma)+\sin^2(\sigma)\sin^2(\sigma_3)+\sin^4(\sigma_3)
	\Big)\\
	&+\frac{\sin^3(\sigma)}{2}
	\Big(\cos(\sigma+\sigma_3)
	-\cos(\sigma-\sigma_3)\cos(2(\sigma+\sigma_3))
	\Big)
=0.
\end{split}
\end{equation}
Obviously $g(\pi-\sigma,\pi-\sigma_3)=g(\sigma,\sigma_3)$.

\begin{proposition}
The bifurcation points between equilateral $LRE$ and
isosceles $LRE$ are $\sigma_k=\sigma_c$ and $\sigma_k = \pi-\sigma_c$,
where $\sigma_c=2^{-1}\arccos(-4/5)=1.249...$ (see Figure \ref{figLREForEqualMasses}).
\end{proposition}
\begin{proof}
The bifurcation points are the solutions of 
\begin{equation}
g(\sigma,\sigma)
=\sin^5(\sigma)\Big(4+5\cos(2\sigma)\Big)
=0.
\end{equation}.
\end{proof}
\begin{figure}
   \centering
   \includegraphics[width=6cm]{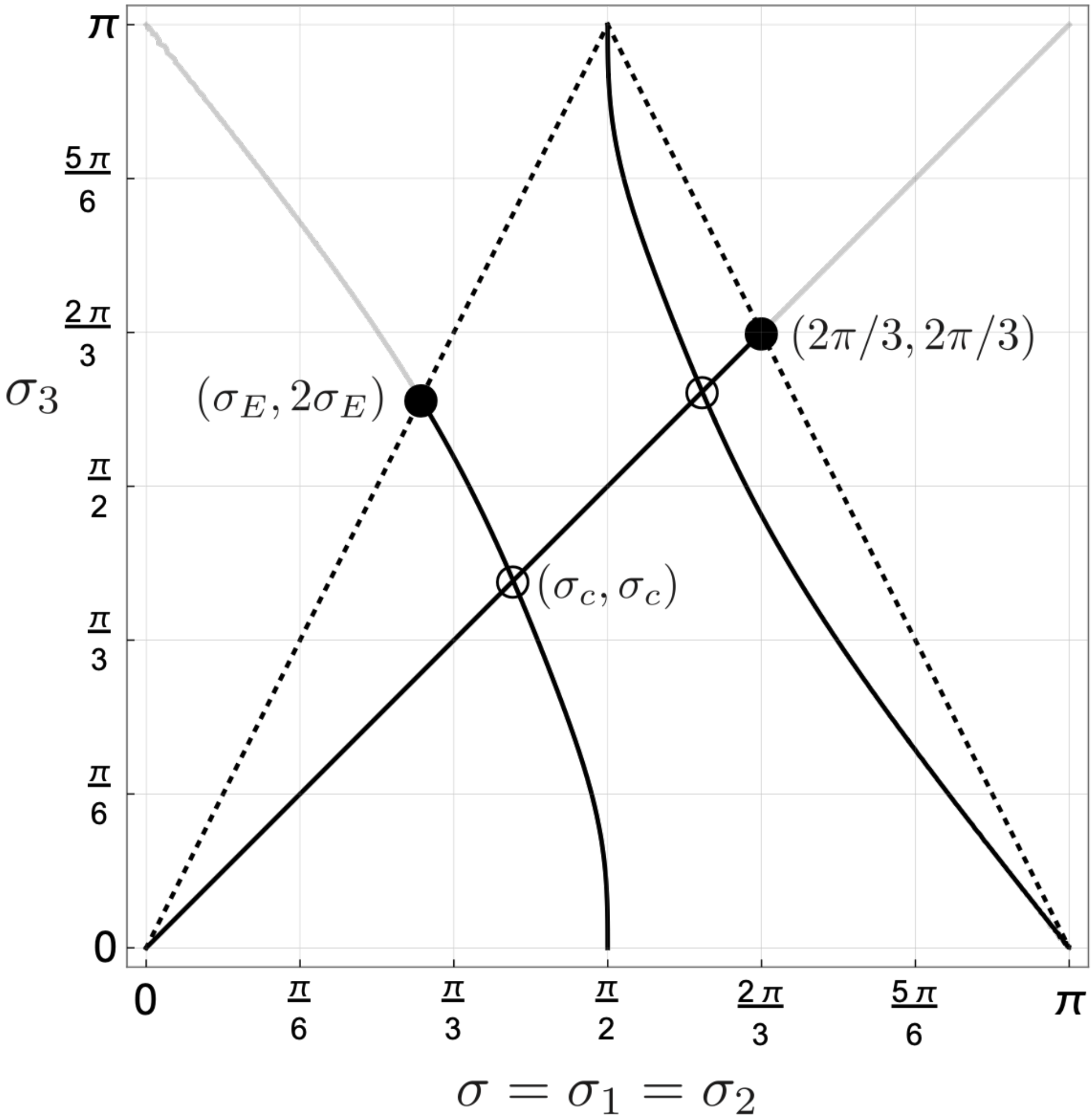} 
   \caption{Equilateral $LRE$ (the straight line)
   and isosceles $LRE$ (two curves).
   Two points with hollow circle are the bifurcation points
   of equilateral and isosceles.
   The dotted lines represents the boundary of $U_\textrm{phys}$.
   The end 
   points shown by the black circles are  Euler $RE$.
   The point $(\sigma_E,2\sigma_E)$ is the bifurcation point
   of $LRE$ and $ERE$.
   }
   \label{figLREForEqualMasses}
\end{figure}

\begin{proposition}\label{eulerEndPointsForEqualMasses}
The end points of an equilateral $LRE$,
$\sigma_k=2\pi/3$, $k=1,2,3$,
is the $ERE$ on the equator.
The 
end point 
of isosceles $LRE$,
$p_E=(\sigma_E,2\sigma_E)$
with 
\begin{equation}\label{defSigmaE}
\sigma_E=\arccos(2^{-3/4})=0.934...
\end{equation}
is the $ERE$ on a rotating meridian.
\end{proposition}
\begin{proof}
The former is obvious.
The latter is the solution of
\begin{equation}
g(\sigma,2\sigma)
=2\cos(\sigma)\sin(\sigma)^5
	\Big(3+2\cos(2\sigma)\Big)
	\Big(2+4\cos(2\sigma)+\cos(4\sigma)\Big)
=0.
\end{equation}
\end{proof}

The point $p_E$ is the bifurcation point 
between isosceles $LRE$ and $ERE$ 
on a rotating meridian.
We will consider this bifurcation point in the next section.

\begin{remark}
Although the points $\sigma_k=2\pi/3$ and $p_E$ are the end points 
of $LRE$ in $U_\textrm{phys}$,
they are not the end points 
of configurations of $LRE$, they are just the coupling points.
Remember the definitions
\ref{defBifurcationPoint} and
\ref{defEulerJunction}.
\end{remark}


\section{Partial equal masses case}\label{partial equal masses}
In this section, we consider the case $m_1=m_2=m_3\nu$,
with $\nu > 0$.

As shown in Property \ref{factorisation},
the condition $\lambda_{12}=0$
is reduced to  $\sigma_1=\sigma_2$ or $\tilde\lambda_{12}=0$.
In the next subsection, we consider the first case,
isosceles $LRE$
(intersection of $\sigma_1=\sigma_2$ and $\lambda_{23}=0$).
The subsection \ref{bifurcationIsoscelesLREandERE}
is devoted to the bifurcation of
isosceles $LRE$ and isosceles $ERE$.
The second case, 
($\tilde\lambda_{12}=0$ and $\lambda_{23}=0$)
describes scalene $LRE$.
In subsection \ref{isoscelesToScalene},
we will treat 
the bifurcation of isosceles $LRE$ and scalene $LRE$
($\sigma_1=\sigma_2$ and
$\tilde\lambda_{12}=0$ and $\lambda_{23}=0$).

\subsection{Isosceles $LRE$}
Let be $\sigma=\sigma_1=\sigma_2$, then $U$ and $U_\textrm{phys}$ take the same form as in \eqref{UandU_phys}.


\subsubsection{Mass ratio for a given shape}
\label{massRatioForGivenShape1}
Let be
\begin{eqnarray}
\alpha(\sigma,\sigma_3)
&=&(1+\cos\sigma_3)\sin^3(\sigma_3)-2\cos(\sigma)\sin^3(\sigma), \\
\beta(\sigma,\sigma_3)
&=&\sin^3(\sigma)-\cos(\sigma)\sin^3(\sigma_3).
\end{eqnarray}

A simple calculation proves that the solutions of $\alpha=\beta=0$
are, the point $(\sigma_E,2\sigma_E)$ in $U_\textrm{phys}$, and the points $(0,0)$ and $(0,\pi)$ on the boundary of $U$. 
Where, $\sigma_E$ is defined by (\ref{defSigmaE}).
See Figure \ref{figRegionDefinedByAlphaBeta}.
\begin{figure}
   \centering
   \includegraphics[width=7cm]{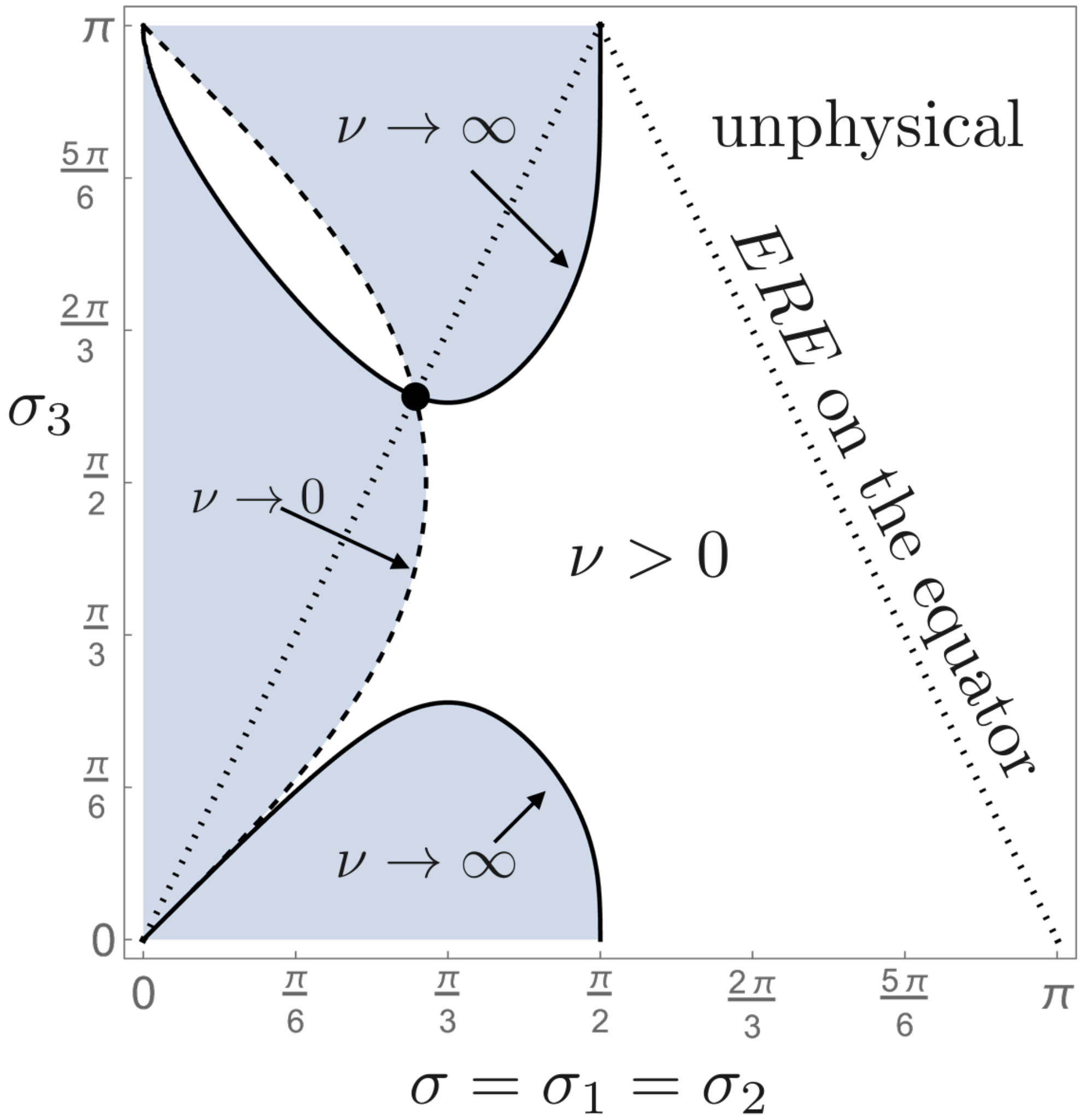} 
   \caption{The contours of $\alpha=0$ (solid curves)
   and $\beta=0$ (dashed curve)
   represent the mass ratio $\nu\to \infty$ and $\nu\to 0$
   respectively.
   The white and grey regions represent 
   $\alpha\beta$ is positive or negative, respectively.
   The two dotted straight lines represents the boundary
   of $U_\textrm{phys}$.
   The dotted line $\sigma_3=2\pi-2\sigma$
   represents isosceles $ERE$ on the equator.
   See Proposition \ref{eulerEndPoints}.
   For any point $(\sigma,\sigma_3)$ in 
   the white region inside the two lines, 
   there is a unique suitable mass ratio $\nu$
   for which isosceles $LRE$ exist.
   In the grey region, isosceles $LRE$ do not exist
   for any choice of positive masses. 
   See Proposition \ref{propForIsoscelesLRE}.
   The point $(\sigma_E,2\sigma_E)$
   shown in the black circle, represents 
    the isosceles $ERE$ on a rotating meridian
   for any mass ratio $\nu$.
   See Propositions \ref{propisoscelesEREforAnyNu}
   and \ref{bifurcationLREandEREforPartialEqualMass}.
   }
   \label{figRegionDefinedByAlphaBeta}
\end{figure}

Then the following proposition follows.
\begin{proposition}
\label{propForIsoscelesLRE}
For $m_1=m_2$,
any point $(\sigma,\sigma_3)$  in $U_\textrm{phys}$
with $\alpha\beta>0$
forms an isosceles $LRE$ by
choosing suitable $\nu$.
\end{proposition}
\begin{proof}
For this case, $\lambda_{12}=0$ is satisfied.
Therefore, the condition for $LRE$ is
\begin{equation}\label{conditionIsoscelesLRE2}
\begin{split}
\lambda_{23}
=-\nu\alpha/\sin^3(\sigma_3)\, 
	+\beta/\sin^3(\sigma)
=0.
\end{split}
\end{equation}
Therefore, if $\alpha\ne 0$ and $\beta\ne 0$,
\begin{equation}\label{NuforIsoscelesSigma}
\nu=\frac{\sin^3(\sigma_3)\beta(\sigma,\sigma_3)}{\sin^3(\sigma)\alpha(\sigma,\sigma_3)}.
\end{equation}
This equation defines $\nu$ in terms of $(\sigma, \sigma_3)$ uniquely,
and $\nu>0$ demands $\alpha\beta>0$. In fact, the region $\alpha\beta>0$ in $U_\textrm{phys}$ is
the region where $\alpha>0$ and $\beta>0$.
\end{proof}

The next result follows inmediatly from equation (\ref{conditionIsoscelesLRE2}).

\begin{proposition}
\label{propisoscelesEREforAnyNu}
The shape $(\sigma_E,2\sigma_E)$ that makes $\alpha=\beta=0$
satisfies the condition for $LRE$ for any $\nu$.
\end{proposition}

Actually this shape is an $ERE$ on a rotating meridian.
Therefore, it corresponds to the bifurcation point of $LRE$ and $ERE$,
that is, we can pass from a $LRE$ to an $ERE$ or vice versa.

\begin{proposition}[Euler coupling]\label{eulerEndPoints}
On the line $\sigma_3=2\sigma$, only $\sigma=\sigma_E$
satisfies the condition for $LRE$ with any $\nu$.
On the other hand, on the line $\sigma_3=2\pi-2\sigma$,
any point in $\pi/2<\sigma<\pi$ satisfies the condition for $LRE$
choosing suitable $\nu<4$.
\end{proposition}
\begin{proof}
On the line $\sigma_3=2\sigma$,
$\alpha=2\sin^3(\sigma)\cos(\sigma)(8\cos^4(\sigma)-1)$
and $\beta=\sin^3(\sigma)(1-8\cos^4(\sigma))$.
Therefore, $\alpha\beta<0$ if $\sigma\ne \sigma_E$,
 and $\alpha=\beta=0$ if $\sigma= \sigma_E$.

On the other hand,
on the line $\sigma_3=2\pi-2\sigma$,
$\alpha=-2\cos(\sigma)\sin^3(\sigma)(1+8\cos^4(\sigma))$
and
$\beta=\sin^3(\sigma)(1+8\cos^4(\sigma))$.
Therefore $\alpha\beta>0$ because $\pi/2<\sigma<\pi$.
For this case, $\nu=4\cos^2(\sigma)<4$.
This shape is the $ERE$ on the equator.
And the inequality $\nu<4$ is exactly the same as the condition (\ref{ConditionEREonEquator})
for  $ERE$ on the equator.
\end{proof}

\subsubsection{Bifurcation point between isosceles $LRE$ and isosceles $ERE$}
\label{bifurcationIsoscelesLREandERE}
\begin{proposition}\label{bifurcationLREandEREforPartialEqualMass}
The shape $(\sigma_E,2\sigma_E)$ is the unique
bifurcation point between isosceles $LRE$ and isosceles $ERE$.
\end{proposition}
\begin{proof}
In the previous Proposition we have already proved that
the shape $p_E=(\sigma_E,2\sigma_E)$ is the Euler 
coupling of isosceles $LRE$.
Obviously, the shape $p_E$ is an $ERE$ on the rotating meridian, and then, continuation from this shape of isosceles $ERE$ for any $\nu > 0 $ does exist.

The proof of the existence of the continuation of a $LRE$ for any $\nu$
is given by showing that $\nabla\lambda_{23}|_{p_E}\ne 0$.
In fact,
\begin{equation}
\begin{split}
&\nabla\lambda_{23}|_{p_E}
=\left.\left(-\frac{\nu\nabla\alpha}{\sin^3(\sigma_3)}
	+\frac{\nabla\beta}{\sin^3(\sigma)}\right)\right|_{p_E}\\
&=\left(
	\frac{2^{5/4}\Big(1+\sqrt{2}+\nu(\sqrt{2}-1)\Big)}{\sqrt{4-\sqrt{2}}},
	\frac{2^{1/4}\Big(3(\sqrt{2}-1)+\nu(5-2\sqrt{2})\Big)}{\sqrt{4-\sqrt{2}}}
	\right)\\
&\ne 0 \mbox{ for any } \nu>0.
\end{split}
\end{equation}
By the implicit function theorem,
there is a continuation of $LRE$ from $p_E$.
\end{proof}


\subsubsection{Isosceles $LRE$ for the restricted three-body problem}
\label{isoscelesLREforRestrictedProblem}
Here we consider the isosceles $LRE$ 
in the restricted problem
with $m_1=m_2$ finite and $m_3\to 0$.
The size dependence will be interesting.

The isosceles $LRE$ for this limit are represented by
the $\nu \to \infty$ curve in  Figure \ref{figRegionDefinedByAlphaBeta},
since $\nu=m_1/m_3=m_2/m_3$.
As we can see,
there are two curves where $\nu \to \infty$,
one connects $(0,0)$ and $(\pi/2,0)$,
and the other connects $(\sigma_E,2\sigma_E)$ and $(\pi/2,\pi)$.

Now, trace the isosceles $LRE$ with respect to $\sigma=\sigma_1=\sigma_2 \in (0,\pi/2)$.
For sufficiently small $\sigma$,
there is one $LRE$ shape that is almost equilateral.
Note that one shape in $(\sigma, \sigma_3)$ represents
four $LRE$ configurations,
two for orientations of the triangle,
and two for the places near the north pole or south pole.
So, there are four $LRE$ configurations for $0<\sigma<\sigma_E$.
%
Then at $\sigma=\sigma_E$,
new $LRE$ shape bifurcated from $(\sigma_E,2\sigma_E)$.
So, there are eight $LRE$ configurations for $\sigma_E<\sigma<\pi/2$.

With respect to $\sigma_3 \in (0,\pi/2)$,
the situation is more complex.
Note that the graph $\alpha(\sigma,\sigma_3)=0,$
takes the local minimum and maximum for $\sigma_3$ at $\sigma=\pi/3$,
where the minimum and maximum value of $\sigma_3$
are the solution of 
$\alpha(\pi/3, \sigma_3)
=(1+\cos(\sigma_3))\sin^3(\sigma_3)=(\sqrt{3}/2)^3$.
The solutions are
$\sigma_3=\sigma_s=0.81...$ (the local maximum),
and $\sigma_\ell=1.84...$ (the local minimum).
So, the range $\sigma_3\in (0,\pi/2)$
is divided into four pieces
by the three values 
$\sigma_3=\sigma_s, \sigma_\ell,2\sigma_E$.
As shown in Figure \ref{figRegionDefinedByAlphaBeta},
in the interval $\sigma_3 \in (\sigma_s,\sigma_\ell)$,
there are no isosceles $LRE$ for the restricted problem.

\subsection{Bifurcation points between isosceles $LRE$ and scalene $LRE$}
\label{isoscelesToScalene}
In this subsection,
we will show the existence of the bifurcations 
between isosceles $LRE$ with $\sigma_1=\sigma_2$ and scalene $LRE$.
As described in the section \ref{shapes},
the bifurcation point of this type
is the intersection of three surfaces,
$\sigma=\sigma_1=\sigma_2$,
$\tilde\lambda_{12}=0$,
and $\lambda_{23}=0$.

Now, since
\begin{equation}
\tilde{\lambda}_{12}(\sigma,\sigma,\sigma_3)
=\sin^2(\sigma)\Big(
	6\nu \sin^3(\sigma)\cos(\sigma)\cos(\sigma_3)
	+(1+2\cos(2\sigma))\sin^3(\sigma_3)
	\Big),
\end{equation}
we obtain that $\tilde{\lambda}_{12}(\sigma,\sigma,\sigma_3)=0$
is equivalent to
\begin{equation}\label{tildeLambda12atIsosceles}
6\nu \sin^3(\sigma)\cos(\sigma)\cos(\sigma_3)
	+(1+2\cos(2\sigma))\sin^3(\sigma_3)
=0.
\end{equation}

On the other hand, $\lambda_{23}(\sigma,\sigma,\sigma_3)=0$
determines $\nu$ by equation (\ref{NuforIsoscelesSigma}).
Substituting this $\nu$ into the equation (\ref{tildeLambda12atIsosceles}),
we get the equation for the set of 
$\tilde\lambda_{12}=\lambda_{23}=0$
and 
$\sigma=\sigma_1=\sigma_2$,
\begin{equation}
\begin{split}
j(\sigma,\sigma_3)
&=(1+2\cos(2\sigma))\sin^3(\sigma_3)
+\frac{6\cos(\sigma)\cos(\sigma_3)
			(\sin^3(\sigma)-\cos(\sigma)\sin^3(\sigma_3))}
	{1+\cos(\sigma_3)-2\cos(\sigma)\sin^3(\sigma)/\sin^3(\sigma_3)}\\
&=0.
\end{split}
\end{equation}

Unfortunately, we don't have a proof that
almost all points on the curve $j(\sigma,\sigma_3)=0$ have a continuation of scalene $LRE$.
But we are able to give a proof for two points.
\begin{example}\label{example1}
The points $p_a=(\pi/3,\pi/2)$ and $p_b=(2\pi/3,\pi/2)$
are bifurcation points between isosceles and scalene $LRE$. 
\end{example}
\begin{proof}
It is not difficult to verify that the points $p_a$ and $p_b$ satisfy the condition $j=0$. In fact 
the point $p_a$ corresponds to an isosceles $LRE$ for $\nu=8(36-5\sqrt{3})/333$.
For  $\nabla=(\partial_{\sigma_1},\partial_{\sigma_2},\partial_{\sigma_3})$,
the vector
$\nabla \tilde\lambda_{12}(\sigma_1,\sigma_2,\sigma_3) \times \nabla\lambda_{23}(\sigma_1,\sigma_2,\sigma_3)
=3\sqrt{3}\,\,\nu/8\, (1,-1,0)$ 
indicates the scalene direction,
and the vector 
$\nabla (\sigma_1-\sigma_2)\times \nabla\lambda_{23}
=(-\nu,-\nu,8/3)$ 
indicates 
the isosceles direction.
By the implicit function theorem,
there are continuation of scalene $LRE$ 
and isosceles $LRE$ from $p_a$.

Similarly, $p_b$ corresponds to an isosceles $LRE$ for 
$\nu=8(36+5\sqrt{3})/333$.
In this case the corresponding vectors are 
$\nabla \tilde\lambda_{12}\times \nabla\lambda_{23}
=3\sqrt{3}\,\,\nu/8\, (-1,1,0)$
and
$\nabla (\sigma_1-\sigma_2)\times \nabla\lambda_{23}
=(-\nu,-\nu,8/3)$.
\end{proof}

In subsection \ref{isoscelesAndScaleneNumerical},
we will give a numerical result
that shows that
there are continuation of scalene $LRE$ from 
the points on $j=0$ except for three exceptional points.

\subsection{Numerical results}
In this subsection we present numerical simulations which show some continuations of $RE$.

\subsubsection{Isosceles $LRE$}

\begin{figure}[h]
   \centering
   \includegraphics[width=5in]{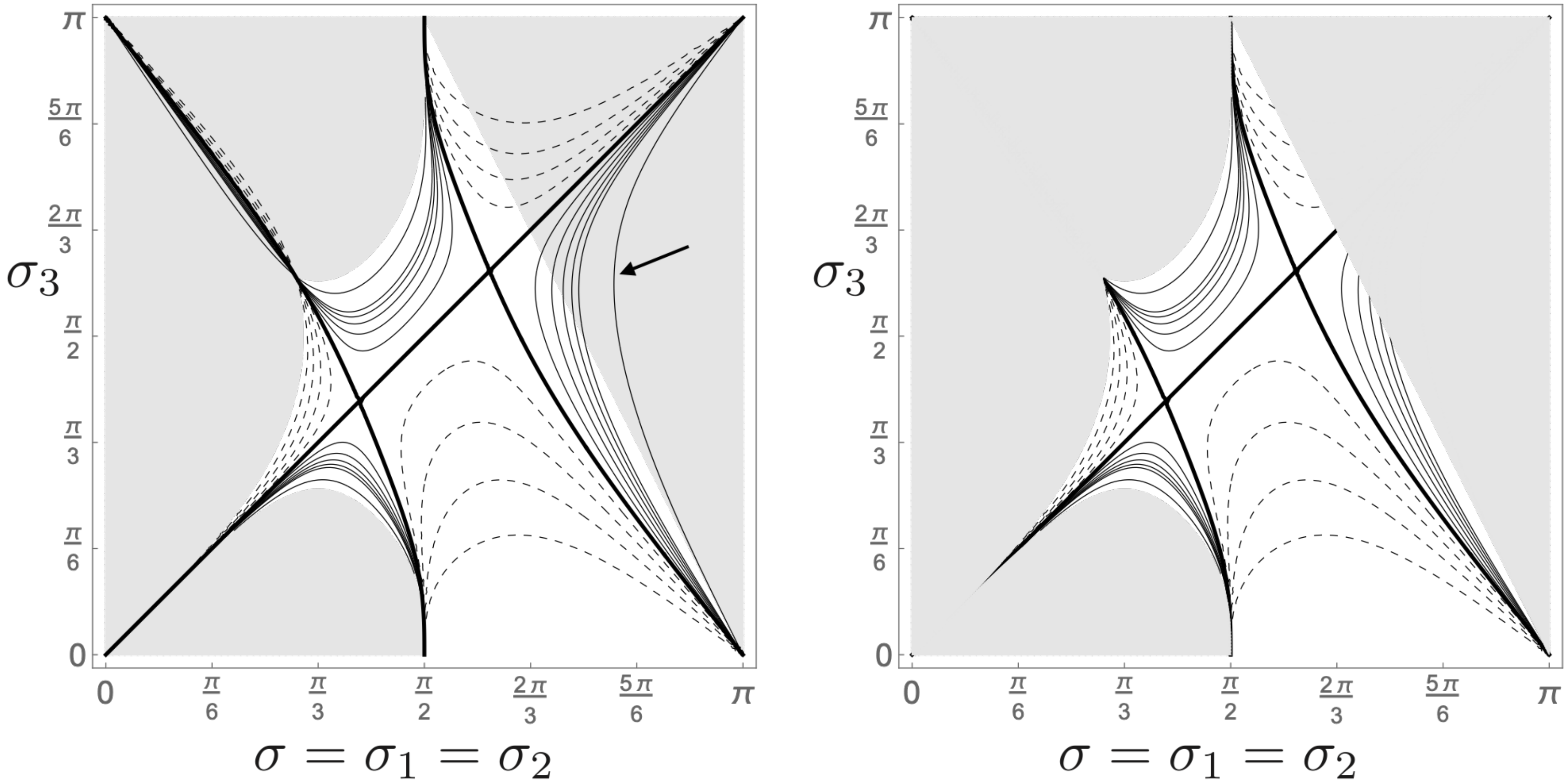}
   \caption{Contours for 
   $\nu=$ 
   a positive constant
   in the $(\sigma,\sigma_3)$ plane.
   The left and right pictures represent the contours
   in $U$ and $U_\textrm{phys}$
   respectively.
   The thick contour represents $\nu=1$.
   The dashed contours are for $\nu=0.2,0.4,...,0.8$,
   and solid  contours are for $\nu=1.2,1.4,...,1.8,2,4$.
   Note that the rightmost contour for $\nu=4$
   (pointed by an arrow)
   in $U$ is outside of 
   $U_\textrm{phys}$.
   }
   \label{figContourForNu}
\end{figure}
The result of the numerical calculations are shown in  Figure
\ref{figContourForNu}.
The  contours represent curves for $\nu$ a positive constant.
The white region in this figure represents $\nu>0$.
The gray region is $\nu<0$ or outside of $U_\textrm{phys}$.
Every  contour with $\nu\ne 1$ 
has a unique continuation from an edge to another edge of 
$U$.

The \olre\  continuations
for $\nu>0$ are emerging from the origin
as showed above,
it looks close to a straight line.
The end points of this continuation are
$(\sigma,\sigma_3)=(\sigma_E,2\sigma_E)$ for $\nu<1$,
and $(\sigma,\sigma_3)=(\pi/2,0)$ (binary collision) for $\nu>1$.

Figure \ref{figContourForNu} shows that
the continuation of $LRE$ for any $\nu>0$ are emerging
form $(\sigma_E,2\sigma_E)$,
which is the unique bifurcation point
between $LRE$ and $ERE$ for $m_1=m_2$ case.

In Figure \ref{figContourForNu},
we can see  the Euler 
couplings
on the line $2\sigma+\sigma_3=2\pi$,
which is the $ERE$ on the equator.
Note that the curve $\nu=4$ 
(the rightmost solid curve
in the left side of Figure \ref{figContourForNu})
is outside of $U_\textrm{phys}$,
as shown in Proposition \ref{eulerEndPoints}.
All the other end points
$(\pi/2,0)$, $(\pi,0)$ and $(\pi/2,\pi)$ correspond to collision.

\subsubsection{Bifurcation between isosceles $LRE$ and scalene $LRE$}
\label{isoscelesAndScaleneNumerical}
\begin{figure}
   \centering
   \includegraphics[width=12cm]{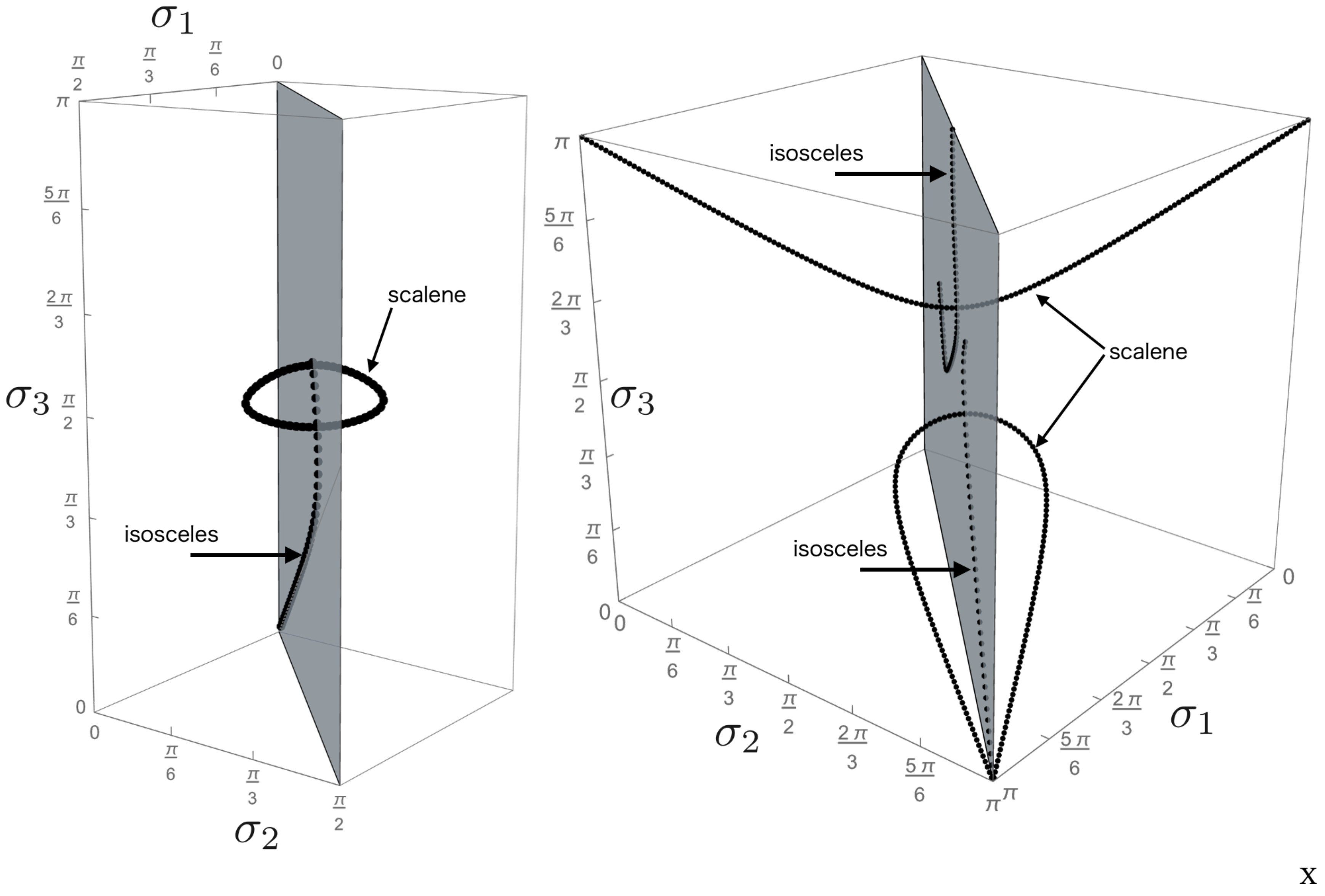}
   \caption{Two examples for bifurcation of isosceles $LRE$ and scalene $LRE$
   for $m_1=m_2\ne m_3$.
   The dotted curves represent the continuation of isosceles $LRE$.
   The solid curves represent the continuations of scalene $LRE$.
   The grey plane is $\sigma_1=\sigma_2$,
   where the isosceles continuations lay on.
   The scalene continuations intersect this plane.
   Left: For $\nu=8(36-5\sqrt{3})/333$, two bifurcation points exist:
   $(\sigma_1,\sigma_2,\sigma_3)=(\pi/3,\pi/3,\pi/2)$,
   and $(0.942...,0.942...,1.850...)$.
   The scalene curve is a loop, and it crosses the isosceles curve
   at two points.
   The isosceles curve is 
   the \olre\ continuation.
   Right:
   For $\nu=8(36+5\sqrt{3})/333$,
   we  also have two bifurcation points
   of isosceles and scalene,
   at $(2\pi/3,2\pi/3,\pi/2)$, and
   $(1.764...,1.764...,2.078...)$.
   For this mass ratio, however, as shown in this picture,
   the bifurcation points are on separate curves.
   }
   \label{figBifurcationOfIsoscelesAndScalene}
\end{figure}

Figure \ref{figBifurcationOfIsoscelesAndScalene} shows
two examples of bifurcations between
isosceles $LRE$ and scalene $LRE$
for $\nu=8(36\pm 5\sqrt{3})/333$
which were described in Example \ref{example1}.
For each mass ratio, two bifurcation pints exists.
For $\nu=8(36-\sqrt{3})/333$,
the continuation of scalene $LRE$ is a closed loop
as shown in the left side of figure \ref{figBifurcationOfIsoscelesAndScalene}.
The loop intersects one continuation curve of the isosceles continuation twice.
This isosceles continuation is 
the \olre\ continuation.
For $\nu=8(36+\sqrt{3})/333$, we also get two bifurcation points.
In this case, two scalene continuation curves intersect
two different continuation of isosceles curves.
See the right side of Figure \ref{figBifurcationOfIsoscelesAndScalene}.

Figure \ref{figBifurcationPointToScalene}
shows the global structure of the bifurcation point of
this type.
The points on the curve $j=0$  
give the bifurcation points
if $c=\nabla \tilde\lambda_{12}\times \nabla\lambda_{23}\ne 0$.
The curve $c=0$ is also shown in the same figure.
\begin{figure}
   \centering
   \includegraphics[height=7cm]{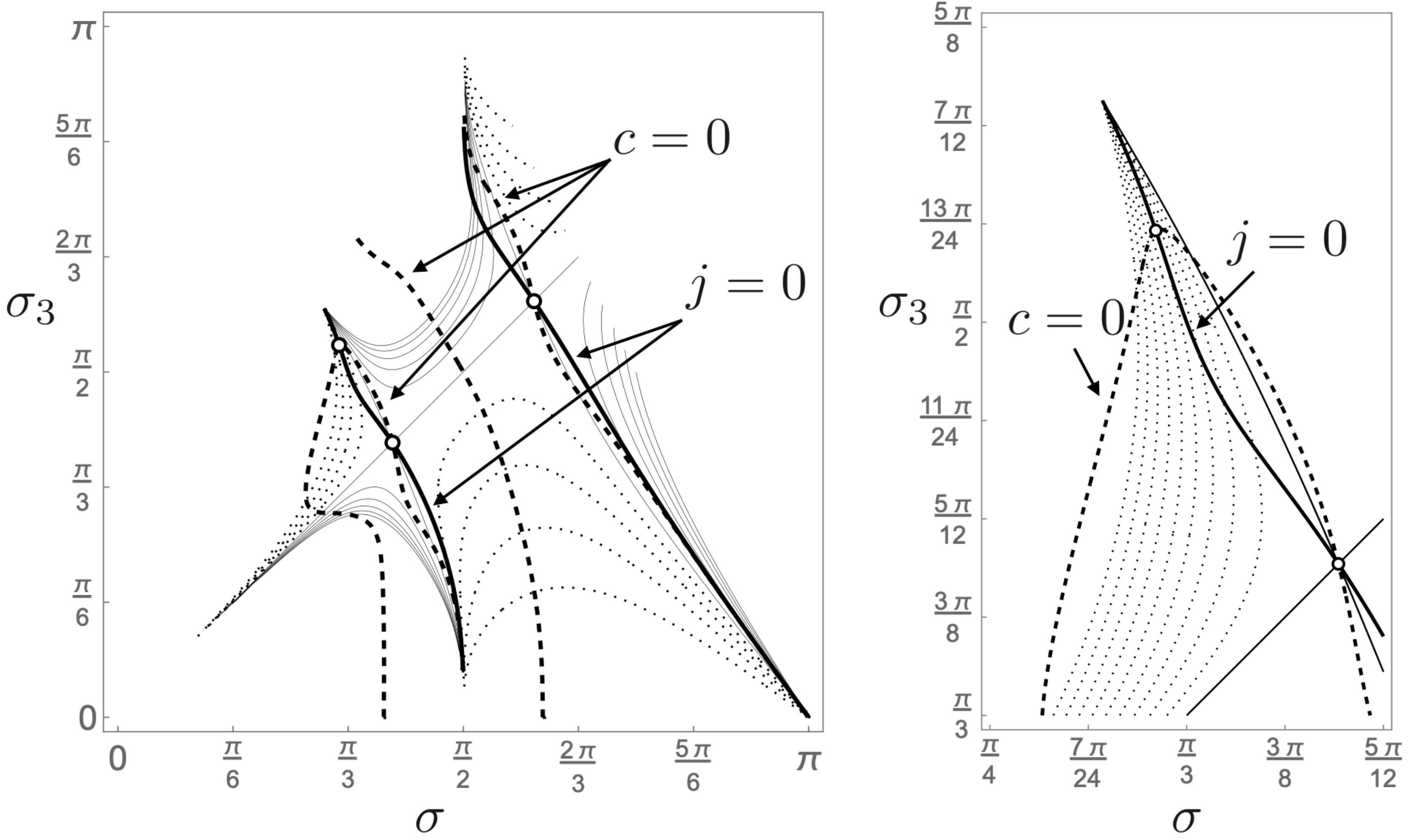}
   \caption{The bifurcation points of
   isosceles $LRE$ ($\sigma=\sigma_1=\sigma_2$)
   and scalene $LRE$.
   The thick curves represent the set of bifurcation points,
   $j=0$.
   Thin dotted  or  solid curves represents $\nu=$ constant contour.
   The points on $j=0$ are the bifurcation point
   if $c=\nabla \tilde\lambda_{12}\times \nabla\lambda_{23}\ne 0$.
   The curve $c=0$ is shown by the thick dashed curves.
   Left: The global view.
   The three points represented by the hollow circle 
   satisfy $j=c=0$ in $U_\textrm{phys}$.
   Right: A close-up view.
   }
   \label{figBifurcationPointToScalene}
\end{figure}

We can see in this figure,
that the bifurcations occur for $|\nu-1|$ sufficiently small. That is, 
the bifurcation from 
\olre\ continuation
is only possible 
if $\nu=1-\delta$ with 
sufficiently small $\delta>0$.
In the following we will show numerically how small could be  $\nu$.

On the cross point of $j=0$ and $c=0$ the bifurcation doesn't occur,
namely, no continuation of scalene solutions exists.
Figure \ref{figBifurcationPointToScalene} shows that
there are three such points.
The points $(\sigma,\sigma_3)=(\sigma_c,\sigma_c)$,
$(\pi-\sigma_c,\pi-\sigma_c)$, correspond to the bifurcation points 
for equilateral and isosceles $LRE$ for $\nu=1$. The other point 
$(\sigma,\sigma_3)=(0.3202...\pi,0.5388...\pi)$ is new.
The mass ratio at this point is $\nu_0=0.6039...$.
Numerical calculations suggest that this point is the lower bound for $\nu$, where we can find bifurcation to scalene $LRE$ from the
\olre\ continuation.
In other words,
the bifurcation of this type occurs for $\nu_0<\nu<1$.
The reason is the following.
At $\nu=8(36-\sqrt{3})/333$, the scalene continuation is a loop.
This means that 
two surfaces of $\tilde\lambda_{12}=0$ and $\lambda_{23}=0$
intersect, and the intersection curve is a loop.
Numerical calculation shows that
smaller $\nu$  makes smaller loop,
and at the limit $\nu\to\nu_0$ the loop becomes just a point
$(0.3202...\pi,0.5388...\pi)$.
Namely, the two surfaces just touch at this point.
Thus the lower bound for $\nu$ in order to have this kind of bifurcation is
$\nu_0=0.6039...$.


\section{$LRE$ with general masses}\label{general masses}
In this section, we will tackle the case of general masses. Remember that if  the three masses are not equal, by Corollary \ref{unequalMassesMakeScalene}, the shapes for 
the $LRE$ are scalene triangles.

\subsection{Continuation from Lagrangian equilateral $RE$ on $\mathbb{R}^2$ to $LRE$ on $\mathbb{S}^2$}\label{sectionLagrangianEquilateral}
Let 
$s_k$ be the arc length between the masses $m_i$ and $m_j$,
where $(i,j,k)=(1,2,3),(2,3,1),(3,1,2)$.
Then $\sigma_k=s_k/R$, where $R$ is the radius of $\mathbb{S}^2$.
The Euclidean limit where the Lagrange equilateral solution exists
is achieved by taking $R \to \infty$.
\begin{proposition}\label{PropCont2}
The Lagrange equilateral solution when $R$ goes to infinity exists,
and the continuation to finite $R$ also exists.
\end{proposition}
\begin{proof}
The limit $R\to \infty$ of 
the conditions $\lambda_{12}=\lambda_{23}=0$
are
\begin{equation}
\frac{(s_1^3-s_2^3)\sum_k m_k s_k^3}{s_1^3s_2^3}
=\frac{(s_2^3-s_3^3)\sum_k m_k s_k^3}{s_2^3s_3^3}
=0.
\end{equation}
The solution is $s=s_k$, $k=1,2,3$, 
which corresponds to the equilateral $LRE$.

Then, for $R\to \infty$,
\begin{equation}
\left.R^{-2}\nabla \lambda_{12}\times \nabla \lambda_{23}
\right|_{\sigma_k=s/R}
\to 
\frac{9(m_1+m_2+m_3)^2}{s^2}\,\,(1,1,1)
\ne 0.
\end{equation}
By the implicit function 
theorem,
the continuation of equilateral $LRE$ to finite $R$ exists.
\end{proof}

The above result was first proved in \cite{Bengochea}, by using a different approach.

\begin{proposition}\label{PropCont3}
The continuation of equilateral $LRE$ in $\mathbb{R}^2$ to $\mathbb{S}^2$ with finite $R$
has $\sigma_i<\sigma_j<\sigma_k$ if $m_i<m_j<m_k$.
\end{proposition}
\begin{proof}

Let be
$\sum_{\ell=1,2,3} \sigma_\ell
=\sum_{\ell=1,2,3} s_\ell/R
=3\epsilon \ll 1$.
Then the expansion of $\sigma_\ell$ to $O(\epsilon^3)$ is
\begin{equation}\label{almostEquilateral}
\sigma_\ell = \epsilon 
	+\left(\frac{m_\ell}{m_1+m_2+m_3}-\frac{1}{3}\right)\frac{\epsilon^3}{3!}
	+O(\epsilon^5).
\end{equation}

Therefore, $\sigma_i<\sigma_j<\sigma_k$ 
if $m_i<m_j<m_k$ for
sufficiently small $\sigma$.
By Proposition \ref{equalSigmaNeedsEqualMass},
the ordering of $\sigma_\ell$ cannot be changed
in the continuation of the solution.
Therefore, 
the ordering $\sigma_i<\sigma_j<\sigma_k$ is preserved 
in the continuation to finite size of $\sigma$.
\end{proof}

\begin{remark}\label{remark2}
For finite $R$, 
there are no equilateral solutions if the masses are not equal.
Instead, there are  almost equilateral solutions if $\sigma_\ell \ll 1$.
Similarly, there are three continuations of $ERE$
which are almost similar to the three Euler solutions
on the Euclidean plane $\mathbb{R}^2$ \cite{Fujiwara}. We call such continuation of the solutions as ``\,\olre\ continuation'' and ``\,\oere\ continuations'' respectively, because these continuations start from
the origin. 
\end{remark}

\subsection{Continuation of Euler $RE$ on the equator
to Lagrange $RE$}
If the masses satisfy the condition (\ref{ConditionEREonEquator}),
then $ERE$ on the equator exists.
\begin{proposition}
The continuation of $LRE$ from 
an $ERE$ on the equator exists.
\end{proposition}
\begin{proof}
Direct calculation shows that $\lambda_{12}=\lambda_{23}=0$
are satisfied by
$\sigma_k$ in
(\ref{SigmaForEREonEquator}).
Besides that, we get
\begin{equation}
\begin{split}
&m_1^2m_2^2m_3^2\,\,(1,1,1)\cdot(\nabla\lambda_{12}\times\nabla\lambda_{23})\\
&=-\left(\sum_k m_k^2\right)
	\left(\sum_\ell \mu_\ell\right)
	(\mu_1+\mu_2-\mu_3)(\mu_2+\mu_3-\mu_1)
	(\mu_3+\mu_1-\mu_2)\\
&\ne 0
\mbox{ by the condition (\ref{ConditionEREonEquator})}.
\end{split}
\end{equation}
By the implicit function theorem, we get the result.
\end{proof}

\begin{proposition}\label{PropContOfEREOnEquator}
The continuation of $LRE$ from the $ERE$ on the equator
has $\sigma_i<\sigma_j<\sigma_k$
if $m_i<m_j<m_k$.
\end{proposition}
\begin{proof}
Without loss of generality, we can assume that $m_1<m_2<m_3$.
Since
$\mu_k=\sqrt{m_1m_2m_3}/\sqrt{m_k}$
for $k=1,2,3$,
$\mu_1>\mu_2>\mu_3$ is obvious.
Then 
$\cos\sigma_i-\cos\sigma_j
=(\mu_i-\mu_j)\Big((\mu_i+\mu_j)^2-\mu_k^2)\Big)/(2\mu_1\mu_2\mu_3)$
for $(i,j,k)=(1,2,3)$, $(2,3,1)$, and $(3,1,2)$
yields 
$\cos\sigma_1>\cos\sigma_2>\cos\sigma_3$,
because any $ERE$ on the equator satisfies $\mu_i+\mu_j>\mu_k$.
Since $\cos\sigma$ is a decreasing function in $0<\sigma<\pi$,
the proposition is proved.
\end{proof}

\subsection{Mass ratio for a given shape}\label{massRatioForGivenShape2}
To go further into the consideration of the bifurcation of
$LRE$ and $ERE$ for general masses,
we treat the conditions for $LRE$, $\lambda_{12}=\lambda_{23}=0$,
as the equations for the mass 
ratios 
$\nu_1=m_1/m_3$ and 
$\nu_2=m_2/m_3$.
This subsection is an extension of the section \ref{massRatioForGivenShape1}
for partial equal masses case
to the general masses case.

The conditions $\lambda_{12}=\lambda_{23}=0$ take the form
\begin{equation}
S\left(\begin{array}{c}\nu_1 \\\nu_2 \\1\end{array}\right)
=0,
\end{equation}
where $S$ is a two by three matrix
in function of $\sigma_k$.
The rank of $S$ is at most two.

If $\rank\, S=2$, $(\nu_1,\nu_2)$ is determined uniquely.
Let be 
\begin{equation}
\tilde{S}=\left(\begin{array}{cc}
S_{11} & S_{12} \\S_{21} & S_{22}
\end{array}\right),\quad
s=\left(\begin{array}{c}S_{13} \\S_{23}\end{array}\right)
\end{equation}
Then
\begin{equation}
{}^t(\nu_1,\nu_2)=-\tilde{S}^{-1}s.
\end{equation}
The following lemmas are obvious.
\begin{lemma}
For a given shape $\sigma_k$,
if $\rank\, S=2$ and if the equation ${}^t(\nu_1,\nu_2)=-\tilde{S}^{-1}s$ gives
$\nu_1>0$ and $\nu_2>0$, then
this shape form a $LRE$ with this mass ratio.
\end{lemma}
\begin{remark}
We have checked the shapes in $U_\textrm{phys}$
with $\sigma_k\in \{\ell\pi/12,1\le\ell<12\}$
and $\sigma_1\le\sigma_2\le\sigma_3$.
There are 73 $LRE$
among the total of $133$ such grid points.
All of them have $\det \tilde S>0$.
There are 25 scalene $LRE$.
For example 
$(\sigma_1,\sigma_2,\sigma_3)=(\pi/4,3\pi/4,5\pi/6)$
has $\det \tilde S=1/4$,
$(\nu_1,\nu_2)=((2+3\sqrt{3})/2,(-2+5\sqrt{3})/2)$,
and $\lambda=2(2+\sqrt{3})$.
\end{remark}

\begin{lemma}
For a given shape $\sigma_k$,
if $\rank\, S=\rank\, \tilde S=1$ and
there are solutions $\nu_k>0$
that satisfy 
$S_{11}\nu_1+S_{12}\nu_2 + S_{13}=0$, then
this shape form 
a $RE$
with these mass ratios.
\end{lemma}

\subsection{Bifurcations between $ERE$ on a rotating meridian
and $LRE$}
If there is a bifurcation point between an $ERE$ on a rotating meridian and a $LRE$, it must be on the plane $\sigma_k=\sigma_i+\sigma_j$.
Without loss of generality, we can take the plane 
$\sigma_3=\sigma_1+\sigma_2$.

\subsubsection{Bifurcation points on the plane $\sigma_3=\sigma_1+\sigma_2$}

The aim of this subsection is to prove
the following result.
\begin{theorem}\label{BifurcationsEREandLRE}
Any point on the curve
\begin{equation}
h=\cos(3\sigma_3)-3\cos(\sigma_3)
	+2\cos(2\sigma_3)\cos(\sigma_1-\sigma_2)
	=0
\end{equation}
which belongs to the plane
$\sigma_3=\sigma_1+\sigma_2$
is a bifurcation point of $ERE$ on a rotating meridian
and a $LRE$ for continuously many $(\nu_1,\nu_2)$.
Inversely,
if a continuation of $LRE$
(for given $\nu_1,\nu_2>0$)
reaches the $\sigma_3=\sigma_1+\sigma_2$ plane,
the point is on the curve $h=0$.
Therefore, this is a bifurcation point
between $LRE$ and $ERE$.
\end{theorem}
The proof of this Theorem is given in a sequence of lemmas.

\begin{lemma}\label{theMeaningOfH}
On the plane $\sigma_3=\sigma_1+\sigma_2$,
$\rank\,  S=\rank\, \tilde S=1$ on the curve
$h=0$,
otherwise
$\rank\,  S=\rank\, \tilde S=2$.
\end{lemma}
\begin{proof}
A direct calculation shows that
\begin{equation}
(S_{11},S_{12},S_{13})\times 
(S_{21},S_{22},S_{23})
=	\left(
		\frac{\sin\sigma_2}{\sin\sigma_1},
		\frac{\sin\sigma_1}{\sin\sigma_2},
		-\frac{\sin\sigma_1\sin\sigma_2}
		{
		\sin\sigma_3
		}
	\right)
	\frac{h}{2}.
\end{equation}
Since 
$0<\sigma_1,\sigma_2,\sigma_3<\pi$,
this vector is null
if and only if $h=0$.
\end{proof}

\begin{lemma}
The curve $h=0$, for $0<\sigma_1+\sigma_2=\sigma_3<\pi$,
is a continuous curve that connects
$(\sigma_1,\sigma_2)=(0,\pi/2)$ and $(\pi/2,0)$.
The range of $\sigma_3$ is
$\pi/2<\sigma_3\le 2\sigma_E=\arccos(-1+1/\sqrt{2})<3\pi/4$.
\end{lemma}
\begin{proof}
Since $0<2\sigma_3<2\pi$,
the solutions of $\cos(2\sigma_3)=0$ are
$2\sigma_3=\pi/2,3\pi/2$.
But $3\cos(\sigma_3)-\cos(3\sigma_3)=\pm 2\sqrt{2}\ne 0$
on these points.
Therefore $\cos(2\sigma_3)\ne 0$ on $h=0$.
Then
\begin{equation}\label{cosDelta}
\cos(\sigma_1-\sigma_2)
	=\frac{3\cos(\sigma_3)-\cos(3\sigma_3)}{2\cos(2\sigma_3)}.
\end{equation}
Since
\begin{equation}
\frac{d}{d\sigma_3}\left(
	\frac{3\cos(\sigma_3)-\cos(3\sigma_3)}{2\cos(2\sigma_3)}
	\right)
=\frac{\big(9+4\cos(2\sigma_3)+\cos(4\sigma_3)\big)\sin(\sigma_3)}
	{2\cos^2(2\sigma_3)}
>0,
\end{equation}
$\cos(\sigma_1-\sigma_2)$ is a strictly increasing function of $\sigma_3$.
Therefore, the upper limit of $\sigma_3$ is given by
$3\cos(\sigma_3)-\cos(3\sigma_3)=2\cos(2\sigma_3)$.
The solution is $\sigma_3=2\sigma_E$ and $\sigma_1=\sigma_2=\sigma_E$.
Decreasing $\sigma_3$, there are two solutions of $(\sigma_1,\sigma_2)$.
At $\sigma_3=\pi/2$, $\cos(\sigma_1-\sigma_2)=0$,
and $(\sigma_1,\sigma_2)=(0,\pi/2)$ and $(\pi/2,0)$.
For $\sigma_3<\pi/2$ the solution $(\sigma_1,\sigma_2)$
are out of 
$0<\sigma_1,\sigma_2$.
Therefore, the range of $\sigma_3$ is $\pi/2<\sigma_3\le 2\sigma_E$.
\end{proof}

\begin{lemma}
On the plane $\sigma_3=\sigma_1+\sigma_2$,
we can find  positive mass ratios that satisfy the conditions 
$\lambda_{12}=\lambda_{23}=0$
only on the curve $h=0$. Outside of this curve,
the conditions demand
$\nu_1,\nu_2<0$.
\end{lemma}
\begin{proof}
On the curve $h=0$, $\rank\,S=\rank\, \tilde  S =1$,
the solutions are given by,
\begin{equation}\label{massRatioOnHeqZero}
\begin{split}
\nu_2
=\nu_1\,
&\frac{1-\cos(\sigma_3)\sin^3(\sigma_1)/\sin^3(\sigma_2)}
	{1-\cos(\sigma_3)\sin^3(\sigma_2)/\sin^3(\sigma_1)}\\
&+\frac{\Big(\cos(\sigma_2)/\sin^3(\sigma_1)-\cos(\sigma_1)/\sin^3(\sigma_2)\Big)
	\sin^3(\sigma_3)}
	{1-\cos(\sigma_3)\sin^3(\sigma_2)/\sin^3(\sigma_1)}.
\end{split}
\end{equation}
As shown above $\cos(\sigma_3)<0$ on $h=0$, therefore
the coefficient of $\nu_1$ in equation \eqref{massRatioOnHeqZero} is always positive; whereas 
the constant term could be positive, zero or negative.
Therefore, for any point on $h=0$,
we can always take positive $\nu_1$ and $\nu_2$.

On the other hand,
at $(\sigma_1,\sigma_2)$ where $h\ne 0$, $\rank\, S=2$.
Therefore, $\nu_1,\nu_2$ are determined uniquely.
A direct calculation shows that both are negative,
\begin{equation}\label{massRatioOnHneZero}
\nu_1=-\sin^2(\sigma_3)/\sin^2(\sigma_1),\quad
\nu_2=-\sin^2(\sigma_3)/\sin^2(\sigma_2).
\end{equation}
\end{proof}
\begin{remark}
As shown in (\ref{massRatioOnHeqZero}),
continuously many $(\nu_1,\nu_2)$
share the same $(\sigma_1,\sigma_2)$ on $h=0$.
\end{remark}

\begin{lemma}
The shape $\sigma_3=\sigma_1+\sigma_2$ on $h=0$
with the mass 
ratios 
given by equation (\ref{massRatioOnHeqZero})
satisfies the condition for $ERE$.
\end{lemma}
\begin{proof}
The condition for $ERE$ with $\sigma_3=\sigma_1+\sigma_2$
is given by equation (\ref{eqForEREI}).
Substitution of the relation (\ref{massRatioOnHeqZero})
into  (\ref{eqForEREI}) yields
\begin{equation}\label{DeqZero}
d=\sin(\sigma_1-\sigma_2)\, h\, r
=0 \quad \mbox{ on  } \quad h=0.
\end{equation}
Where 
\begin{equation}
\begin{split}
r=&\frac{\cos(\sigma_3)\big(1-\cos(2\sigma_3)+\nu_1(1-\cos(2\sigma_1)\big)
\big(\cos(\sigma_3)-\cos(2\sigma_3)\cos(\sigma_1-\sigma_2)\big)}
{4\big(1-\cos(\sigma_3)\sin^3(\sigma_2)/\sin^3(\sigma_1)\big)
\sin^4(\sigma_1)\sin^2(\sigma_2)\sin^2(\sigma_3)}
\end{split}
\end{equation}
is finite on $h=0$ because $\cos(\sigma_3)<0$.
\end{proof}

\begin{lemma}
For a given point $(\sigma_1,\sigma_2)$ on $h=0$ and
$\nu_1,\nu_2>0$ which are related by equation (\ref{massRatioOnHeqZero}),
the continuation of $ERE$ 
from this point exists.
\end{lemma}
\begin{proof}
The substitution of $\nu_2$ given by equation  (\ref{massRatioOnHeqZero})
into $\nabla d=(\partial_{\sigma_1},\partial_{\sigma_2})d$ gives
\begin{equation}
{}^t (\nabla d)
=M\left(\begin{array}{c}\nu_1 \\1\end{array}\right),
\end{equation}
where $M$ is a two by two matrix of functions of $(\sigma_1,\sigma_2)$.
Now, a direct calculation using $h=0$ yields
\begin{equation}
\begin{split}
\det M 
&= \frac{8\sin(\sigma_1-\sigma_2)\sin^5(\sigma_3)\big(3+3\cos(2\sigma_3)+\cos(4\sigma_3)\big)}
{\cos^2(2\sigma_3)\sin^4(\sigma_1)\sin^4(\sigma_2)
\big(1-\cos(\sigma_3)\sin^3(\sigma_2)/\sin^3(\sigma_1)\big)}\\
&\ne 0
\mbox{ on } h=0 \mbox{ and }\sigma_1\ne\sigma_2.
\end{split}
\end{equation}
Therefore, if $\sigma_1\ne\sigma_2$, $\nabla d\ne 0$.
By the implicit function theorem,
there is a continuation of 
$ERE$
from this point.

For the case $\sigma_1=\sigma_2$, we know from by Proposition \ref{equalSigmaNeedsEqualMass} that
$m_1=m_2$.
The existence of the continuation of $ERE$ with $\sigma_1=\sigma_2$
is obvious.
\end{proof}

\begin{lemma}\label{continuationLREexists}
For a given point $(\sigma_1,\sigma_2)$ on $h=0$ and
$\nu_1,\nu_2>0$ that are related by equation (\ref{massRatioOnHeqZero}),
the continuation of 
$LRE$ 
from this point exists.
\end{lemma}
\begin{proof}
Consider $c=\nabla\lambda_{12}\times\nabla\lambda_{23}$
at the point $(\sigma_1,\sigma_2)$,
where $\nabla=(\partial_{\sigma_1},\partial_{\sigma_2},\partial_{\sigma_3})$.
Substituting $\nu_2$ given by equation (\ref{massRatioOnHeqZero}),
we obtain that each component of $c$ has the form $c_i=a_{i1}\nu_1^2+a_{i2}\nu_1+a_{i3}$,
where the $a_{ij}$ are functions of $(\sigma_1,\sigma_3)$.
Namely,
\begin{equation}
A\left(\begin{array}{c}\nu_1^2 \\\nu_1 \\1\end{array}\right)=c.
\end{equation}
Substituting $\sigma_3=\sigma_1+\sigma_2$ and 
using $\cos(\sigma_1-\sigma_2)$ in (\ref{cosDelta}),
we get
\begin{equation}
\begin{split}
\det A
=\frac{\sin^{10}(\sigma_3)}
	{4\sin^7(\sigma_1)\sin^7(\sigma_2)\cos^2(2\sigma_3)}\,\,
	\frac{\sin(\sigma_1-\sigma_2)\,\,n}
		{\big(1-\cos(\sigma_3)\sin^3(\sigma_2)/\sin^3(\sigma_1)\big)^3},
\end{split}
\end{equation}
where
\begin{equation}
\begin{split}
n=&-710 - 1108 \cos(2 \sigma_3) - 449 \cos(4 \sigma_3) - 
  100 \cos(6 \sigma_3)\\
  	&+ 6 \cos(8 \sigma_3) + 8 \cos(10 \sigma_3) + 
  \cos(12 \sigma_3).
  %
%
\end{split}
\end{equation}
We can show that $n\ne 0$  on $h=0$
by a simple calculation.

Therefore, $c\ne 0$ if $\sigma_1 \ne \sigma_2$.
By the implicit function theorem, there is a continuation of
$LRE$ from this point.

For the case $\sigma_1=\sigma_2$, we obtain 
$m_1=m_2$ by Proposition \ref{equalSigmaNeedsEqualMass}.
For this case, $\sigma_1=\sigma_2=\sigma_E$.
The existence of the continuation of $LRE$ from this point
is shown in Proposition \ref{bifurcationLREandEREforPartialEqualMass}.
\end{proof}

The lemmas \ref{theMeaningOfH} to \ref{continuationLREexists}
prove Theorem \ref{BifurcationsEREandLRE}.

\subsubsection{Number of bifurcation points  on $\sigma_3=\sigma_1+\sigma_2$ for given mass ratio}
In the previous Theorem we showed that
any point on the curve $h=0$ located on the plane $\sigma_3=\sigma_1+\sigma_2$ is a bifurcation point between $ERE$ and $LRE$
for continuously many $(\nu_1,\nu_2)$.

Then for given $(\nu_1,\nu_2)$,
how many bifurcation points exist on the plane,
and what $(\sigma_1,\sigma_2)$ if exist?
In this subsection,
we will give an equation to count the number,
and to find the position.

Substituting (\ref{cosDelta}) into (\ref{massRatioOnHeqZero}),
we get
\begin{equation}
\sin\Delta=
\frac{-\delta\nu \Big(16 + 11 \cos(2 \sigma_3) + \cos(6 \sigma_3)\Big) 
\sin(\sigma_3)}
{
 2\cos(2 \sigma_3) \Big( (1 + \cos(4 \sigma_3)) 
 	- \nu (5 + 4 \cos(2 \sigma_3) + \cos(4 \sigma_3))
       \Big)},
\end{equation}
where $\Delta=\sigma_2-\sigma_1$,
$\nu=(\nu_1+\nu_2)/2$ and $\delta\nu=(\nu_2-\nu_1)/2$.
Then,
$1-\cos^2(\Delta)=\sin^2(\Delta)$
with (\ref{cosDelta}) for $\cos\Delta$
will determine $\sigma_3$.
Then $\sin\Delta$ and $\cos\Delta$ will determine $\sigma_1$ and $\sigma_2$.

Note that $1-\cos^2(\Delta)$ with (\ref{cosDelta}) is
a decreasing function of $\sigma_3$, which takes the value 
$1$ for $\sigma_3=\pi/2$ and $0$ for $\sigma_3=2\sigma_E$.
On the other hand,
$\sin^2(\Delta)=(\delta\nu)^2/(1-\nu)^2$ for  $\sigma_3=\pi/2$
 and 
$(13+16\sqrt{2})\delta\nu^2/(2-3\nu)^2$
for $\sigma_3=2\sigma_E$.
Therefore, it may have $0,1,2$ or $3$ solutions of $\sigma_3$.

Then, we get the following lemma.
\begin{lemma}
For the case $|\delta\nu|<|1-\nu|$,
there is at least one bifurcation point between 
 $LRE$ and $ERE$
on the plane $\sigma_3=\sigma_1+\sigma_2$,
$\pi/2<\sigma_3\le2\sigma_E$.
\end{lemma}
\begin{proof}
The denominator of $\sin\Delta$ has zero
if $2/3\le \nu\le 1$.
If $\nu$ is outside of this range,
the result is obvious because $\sin\Delta$ is a continuous function.
If  $\nu$ is in this range,
the function $\sin^2(\Delta)$ is not continuous
but goes to plus infinity at the point.
Therefore, the equation $1-\cos^2(\Delta)=\sin^2(\Delta)$ has
at least one solution.
\end{proof}
\begin{figure}
   \centering
   \includegraphics[width=12cm]{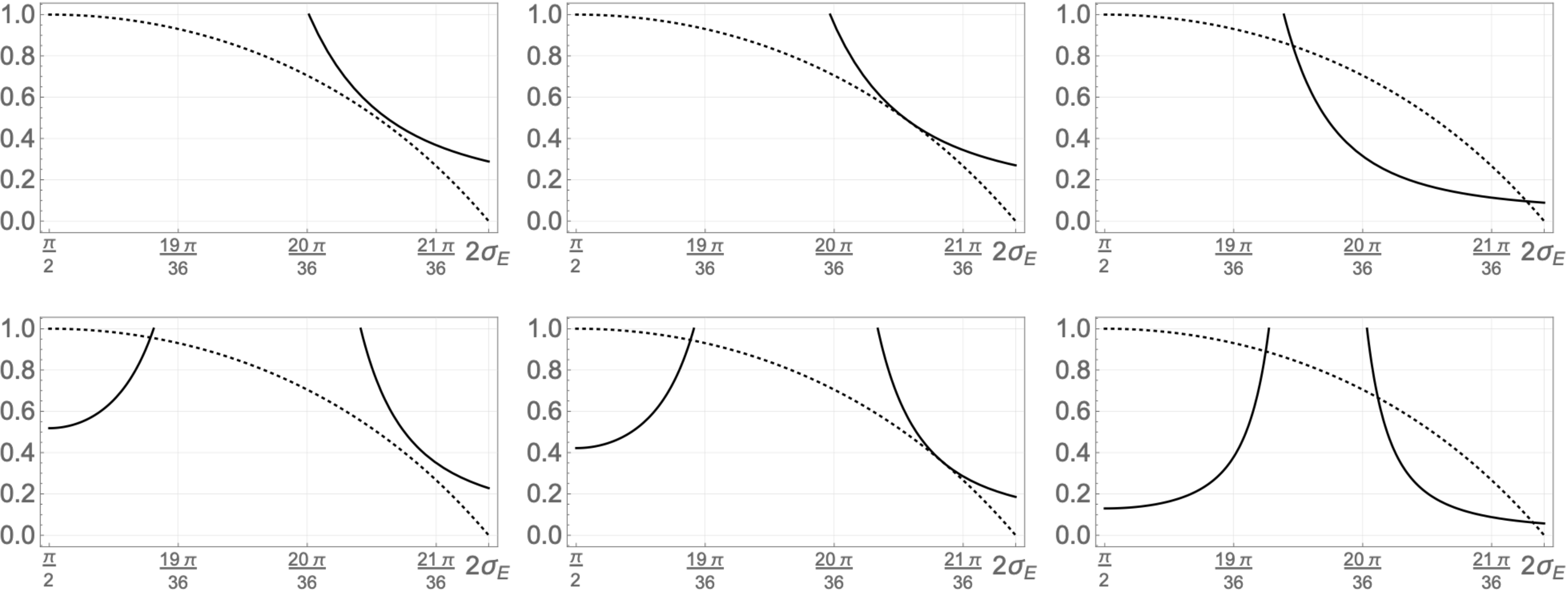}
   \caption{
   The dashed curve and solid curve represents
   $1-\cos^2(\Delta)$ and $\sin^2(\Delta)$ respectively,
   the horizontal axis is $\sigma_3$.
    Upper row from left to right:
	$(\nu,\delta\nu)$ are
	$(1,0.09)$, $(1,0.0870...)$, $(1,0.05)$.
	The number of bifurcation points (intersections) are
	$0,1,2$.
	Lower row from left to right:
	$(\nu,\delta\nu)=(11/12,0.06),(11/12,0.0541...),
	(11/12,0.03)$.
	The number of bifurcation points are $1,2,3$.
    }
   \label{figNumberOfBifurcationPointsOnSigma3}
\end{figure}
Then, the following result is obvious.
\begin{proposition}\label{theHeaviestMass}
For $m_i<m_j<m_k$,
there is at least one bifurcation point
of $LRE$ and $ERE$ 
on the plane $\sigma_k=\sigma_i+\sigma_j$, $\pi/2<\sigma_k\le2\sigma_E$.
\end{proposition}
\begin{proof}
Without loss of generality, we can assume $m_1<m_2<m_3$.
Then $m_2-m_1<2m_3-(m_1+m_2)$,
namely $|\delta \nu|<|1-\nu|$ is satisfied.
\end{proof}

\subsection{Numerical Results}
Figure \ref{figNumberOfBifurcationPointsOnSigma3}
shows  examples of $(\nu,\delta\nu)$
that give $0,1,2,3$ bifurcation points.
Although the case $(\nu,\delta\nu)=(1,0.09)$ has no bifurcation point
on the plane $\sigma_3=\sigma_1+\sigma_2$,
it has the bifurcation point on the plane $\sigma_2=\sigma_3+\sigma_1$
by Proposition \ref{theHeaviestMass}
because $m_2$ is the heaviest for this case.

\begin{figure}[htbp] 
   \centering
   \includegraphics[width=12cm]{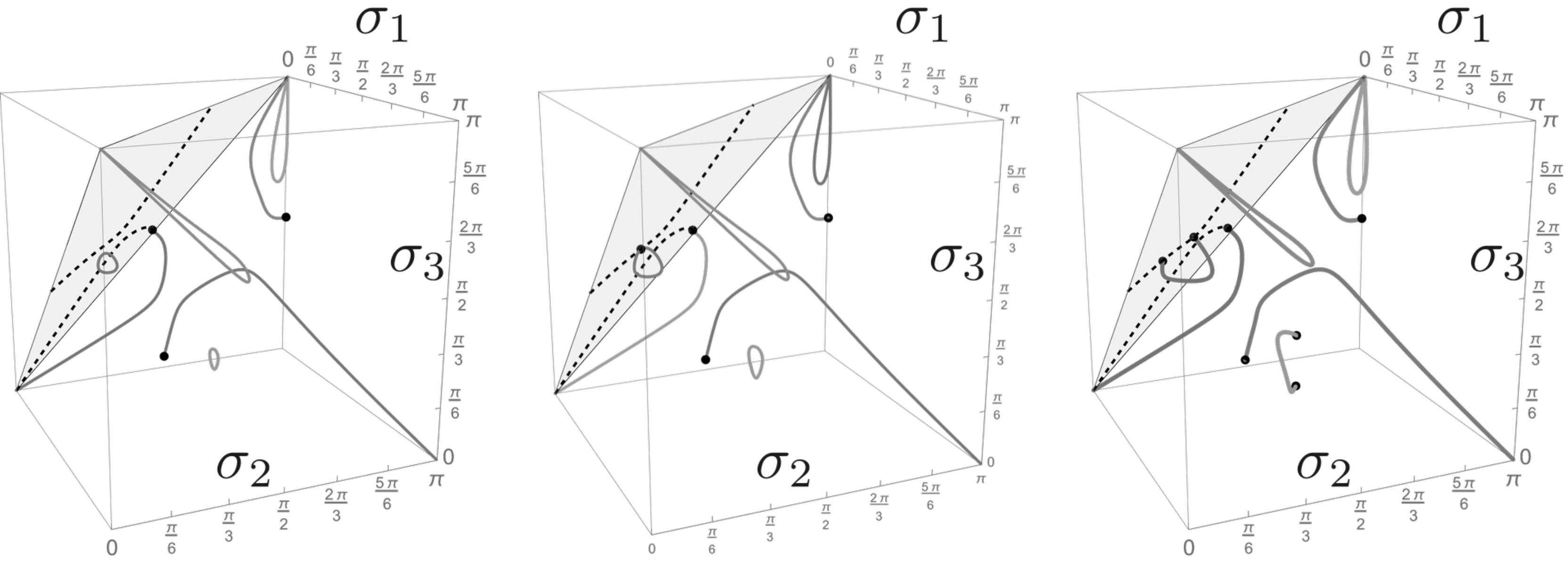}   \caption{Continuations of $LRE$ 
   on $\mathbb{S}^2$
   for $(\nu,\delta\nu)=(11/12,0.06)$(left),
   $(11/12,0.0541...)$(middle)
   and for $(11/12,0.03)$(right).
  They correspond to the lower diagrams
   in Figure \ref{figNumberOfBifurcationPointsOnSigma3}.
   Solid curves represents the continuations of $LRE$.
      There are seven continuations
     of $LRE$. 
     The black balls represent
     bifurcation points between $LRE$ and $ERE$.
     The grey plane corresponds to $\sigma_3=\sigma_1+\sigma_2$
   and the dashed curves on this plane are 
   $ERE$.
   We observe one bifurcation point between $LRE$ and $ERE$ on this plane (subfigure in the left side), two (subfigure in the middle),
   and three (subfigure in the right side).
   }
   \label{figContinuationsOfLRE}
\end{figure}

The corresponding continuations of 
$LRE$ for
$(\nu,\delta\nu)=(11/12,0.06)$, $(11/12,0.0541...)$, and 
$(11/12,0.03)$ 
are shown in Figure \ref{figContinuationsOfLRE}.
The grey plane corresponds to $\sigma_3=\sigma_1+\sigma_2$.
Left: At $(\nu,\delta\nu)=(11/12,0.06)$,
there is only one bifurcation point
between \olre{} and \oere.
As you can see in Fig. \ref{figContinuationsOfLRE}, 
there is a small loop near the plane, which does not reach the plane.
Therefore, the loop doesn't yields a bifurcation point.
As $\delta\nu$ is smaller,
the loop is larger.
Middle: When $\delta\nu=0.0541...$, the loop touches the plane
at one point,
which is a new bifurcation point.
Thus, there are two  bifurcation points for $(11/12,0.0541...)$
as shown in Figure \ref{figNumberOfBifurcationPointsOnSigma3}.
The bifurcated $ERE$ is different from the \oere.
Decreasing $\delta\nu$ the loop gets bigger and then
intersects the plane.
Right: Thus, we have three bifurcation points at $(11/12,0.03)$.

In Figure \ref{figContinuationsOfLRE},
another similar small loop (left) and open arc (right)
which bifurcate to $ERE$ on the plane $\sigma_2=\sigma_3+\sigma_1$
is shown.

The continuation which starts at $(\pi,\pi,0)\in U$ 
bifurcates to $ERE$ on $\sigma_1=\sigma_2+\sigma_3$.
The end point of the continuation
that starts at $(0,\pi,\pi) \in U$
is $ERE$ on the equator.

We see other two loops that start and end
at $(0,\pi,\pi)$ or $(\pi,0,\pi)$. They are elements of the seven continuations in Figure \ref{figContinuationsOfLRE},
where the mass ratios $m_1:m_2:m_3$ are not so far from $1:1:1$.

\begin{figure}
   \centering
   \includegraphics[width=10cm]{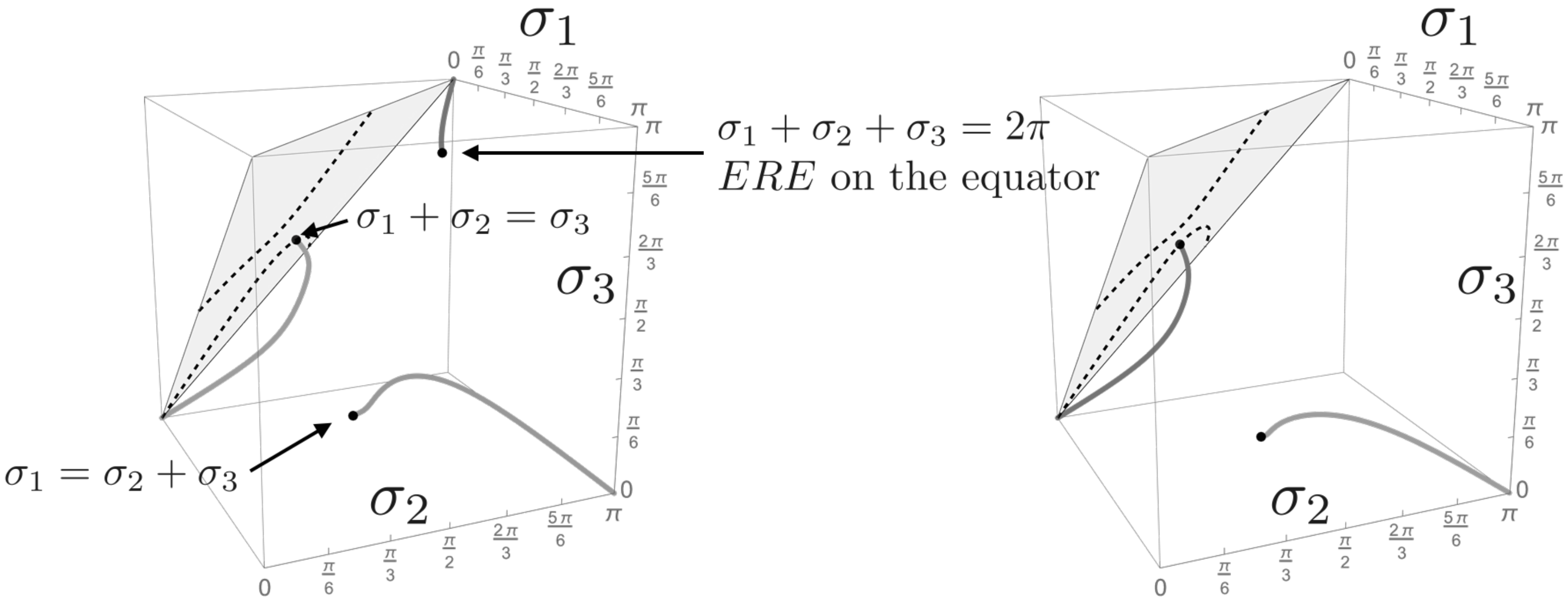}
   \caption{Continuations of $LRE$ and 
   continuations of
   $ERE$ on $\sigma_3=\sigma_1+\sigma_2$ plane.
   Left: For $m_1:m_2:m_3=1:2:4$,
   namely $(\nu,\delta\nu)=(3/8,1/8)$.
   The continuation from $(0,\pi,\pi)\in U$ ends at the $ERE$ on the equator.
   Right: For $m_1:m_2:m_3=1:2:12$,
   $(\nu,\delta\nu)=(1/8,1/24)$.
   At this mass ratio there is no $ERE$ on the equator. 
   So there is no continuation from it.
   }
   \label{figContinuationsForm1m2m31:2:4}
\end{figure}

Figure \ref{figContinuationsForm1m2m31:2:4}
shows the continuations for
$m_1:m_2:m_3=1:2:4$
and $1:2:12$.
No  loop continuations are seen.
Three continuations exist for $m_1:m_2:m_3=1:2:4$.
On the other hand,
since the mass ratio 
$1:2:12$
does not satisfy the condition
for $ERE$ on the equator (\ref{ConditionEREonEquator}),
the $ERE$ and the continuation 
of $LRE$
from this point does not exist.
Thus, there are only two continuations
of $LRE$.

The numerical experiments show a couple of interesting properties
of the continuation of $LRE$, $ERE$, and the bifurcation between them.

(1) The \olre\ continuation and one of
the \olre\ continuations are directly connected
by the bifurcation point.
The \olre continuation
reaches the plane $\sigma_3=\sigma_1+\sigma_2$,
then bifurcates to $ERE$,
where $m_3$ (the heaviest mass) is placed in the middle,
Then it continues  back to  
the \oere\ continuation 
keeping $m_3$ (the heaviest mass)  in the middle.
See Figures \ref{figContinuationsOfLRE} 
and \ref{figContinuationsForm1m2m31:2:4}.
Note that Proposition \ref{theHeaviestMass}
ensures the existence of such bifurcation point between
a $LRE$ and an $ERE$ continuations,
but not ensures that the continuations are \olre\ and \oere.

\medskip

(2)  There are at most $7$ continuations of $LRE$.
Among them, 
the \olre\ continuation
and the continuation of  $LRE$ from $ERE$ on the equator
sharing the same ordering of $\sigma_\ell$,
namely $\sigma_1<\sigma_2<\sigma_3$ if $m_1<m_2<m_3$
by Propositions \ref{PropCont3} and \ref{PropContOfEREOnEquator}.
The other $5$ continuations have mutually different ordering.
Therefore, all possible $6$ orderings of 
$\sigma$'s 
are realised
by the $7$ continuations.


\section{Conclusions and final comments}\label{conclusions}
The variables $(\sigma_1,\sigma_2,\sigma_3)\in U_\textrm{phys}$
for having $LRE$ on $\mathbb{S}^2$
is determined by the two conditions $\lambda_{12}=\lambda_{23}=0$.
Therefore, $LRE$ have one-dimensional intersection curve, namely one-dimensional continuation, in general.
Similarly, $ERE$ on $\mathbb{S}^2$ have one-dimensional continuation.

We have proved the local existence of the continuations.
The proof for the global existence of such continuation still need a better understanding of the surfaces defined by $\lambda_{ij}=0$.

Special attention was paid to 
the \olre\ continuation.
On the Euclidean plane $\mathbb{R}^2$,
there are two isolated Lagrangian equilateral configurations that 
have opposite orientation,
and three isolated Euler configurations.
We have shown that on $\mathbb{S}^2$,
with almost all mass ratios,
the \olre\ continuation
(two almost equilateral configurations
with opposite orientations) and
one \oere\ continuation,
where the heaviest mass is placed between the other two masses,
are connected by the continuations
via bifurcation(s).
See Figure~\ref{figContinuationAndConfigurations}.
This is true for the mass ratios
$m_i\le m_j<m_k$ and $m_i=m_j=m_k$,
but not true for $m_i<m_j=m_k$.
See the $\nu>1$ continuation in Figure \ref{figContourForNu}.
\begin{figure}
   \centering
   \includegraphics[width=13cm]{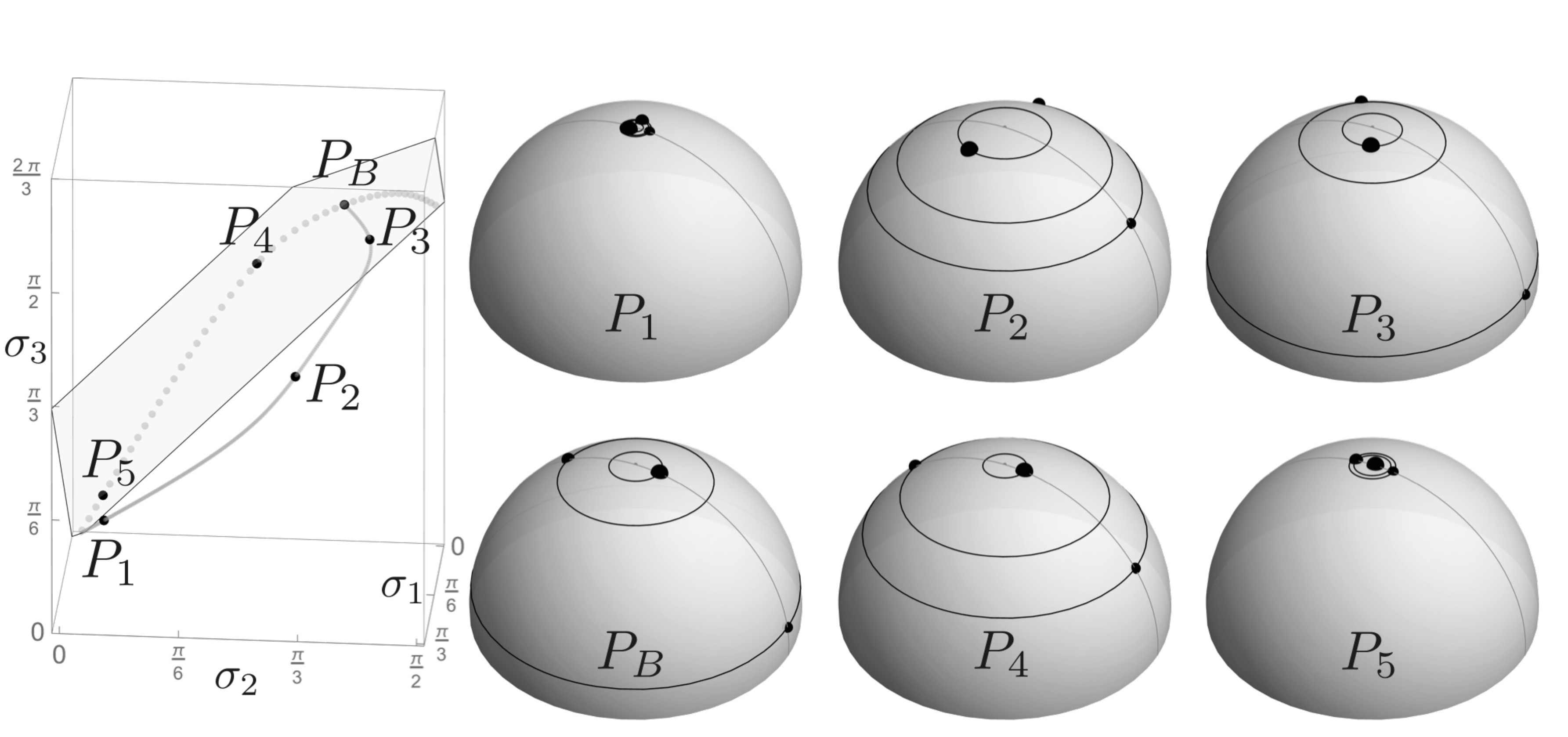}
   \caption{Two continuations from the neighbour of the
   origin in $U_\textrm{phys}$
   (left),
   i.e.
   the \olre\ continuation and 
   \oere\ continuation,
   and the corresponding configurations (right),
   for the mass ratios $m_1:m_2:m_3=1:2:4$. 
   The points $P_1,P_2, P_3$ and $P_B, P_4,P_5$
   represents $LRE$ and $ERE$ respectively.
   The point $P_B$ is the bifurcation point between
   $LRE$ and $ERE$.
   Left: The solid and dotted curve represents
   the \olre\ continuation and \oere\   
   continuation respectively.
   The grey plane is $\sigma_3=\sigma_1+\sigma_2$ plane.
   Right: The configurations
   on the northern hemisphere.
   The meridian of $\phi=0$ and $\pi$ are shown.
   The smallest and the biggest black points represents $m_1$ and $m_3$.
   The mass $m_1$ is placed on $\phi=0$.
   The circles are the orbits of the three bodies.
   }
   \label{figContinuationAndConfigurations}
\end{figure}

When we say that there are two Lagrange equilateral 
solutions in the planar three-body problem, we count the similarity class of the shapes under rotations and homotheties. Since we don't have these similarity classes on $\mathbb{S}^2$,
we have continuously many $LRE$
as shown in Figure \ref{figContinuationAndConfigurations}.
However, the counting based on similarity
has not much meaning
to compare the number of solutions on $\mathbb{S}^2$
and Euclidean $\mathbb{R}^2$.

It is better to count the number of $RE$
on the basis of continuity instead of similarity.
Because the similar shape $r_{ij}=\lambda a_{ij}$
with fixed $a_{ij}$ and scaling factor $\lambda$
is an one dimensional continuation,
the continuation is a natural extension of similarity.
For the continuation,
we have one $LRE$ in the space $r_{ij}$, and
two in the configuration space counting two orientations.
By the same counting,
we have at most seven continuations in $U_\textrm{phys}$,
and $2\times2\times 7=28$ continuations for configurations.
Where, the first factor $2$ counts the orientations,
and the second $2$ counts whether the placement is in the northern or southern hemisphere ($s=\pm 1$ in equation (\ref{sigmaToTheta})).

In this article we have considered only positive masses. Some authors have studied the case for two positive masses and a third massless particle, called the restricted three body problem on $\mathbb{S}^2$,
see for instance \cite{Kilin, Mtz}. 
An interesting question is try to extend our results to the restricted problem.  In this paper we have analyzed just the restricted isosceles $LRE$ on the sphere, that is, we consider $m_1=m_2>0$ and $m_3=0$ (see  
subsection \ref{isoscelesLREforRestrictedProblem}). We showed that our results cover this particular case. We pointing out that in this case, there are values of $\sigma_3$ not allowed 
for having $LRE$, size shape dependence is interesting. However we are far to have a complete analysis of the general restricted three body problem on the sphere. It will be part of another paper. 

Another important 
question is about the stability of the relative equilibria that we found in this work. To tackle this problem we 
need to use geometric mechanics techniques as for instance in \cite{Bolsinov1}, we will do this in a future work. Last but not least, we point out that by using our techniques, we can also study the relative equilibria for the vortex problem on the sphere as in \cite{Bolsinov2}.

\subsection*{Acknowledgements}
The authors of this paper thank to the anonymous reviewer for his/her comments and suggestions, which help us to improve the manuscript. The second author (EPC) has been partially supported 
by Asociaci\'on Mexicana de Cultura A.C. and Conacyt-M\'exico Project A1S10112.

\end{document}